%% file: projection.tex
\documentclass[10pt,a4paper]{amsart} 

\usepackage[utf8]{inputenc}
\usepackage[T1]{fontenc}
\usepackage{amsmath, amssymb}
\usepackage{amsthm}
\usepackage{mathtools}   

\usepackage{xspace} 
\usepackage{xfrac}  
\usepackage{ifthen} 
\usepackage{xifthen} 
\usepackage{array} 
\usepackage{booktabs} 

\usepackage{enumitem} 
\setenumerate[0]{label=(\roman*)}   

\usepackage{hyperref}
\hypersetup{
    unicode=true,          
    pdftitle={Orthogonal projections of discretized sets},    
    pdfauthor={He Weikun},     
    pdfsubject={Arithmetic combinatorics},   
    pdfkeywords={Discretized}{Projection}, 
    pdfborder={0 0 0}
}

\renewcommand{\phi}{\varphi} 

\newcommand\dash{\nobreakdash-\hspace{0pt}}

\newtheorem{thm}{Theorem}
\newtheorem{coro}[thm]{Corollary}
\newtheorem{prop}[thm]{Proposition}
\newtheorem{lemm}[thm]{Lemma}

\theoremstyle{definition}

\def\N{{\mathbb N}}
\def\Z{{\mathbb Z}}
\def\R{{\mathbb R}}

\def\dd{{\,\mathrm{d}}}

\def\AC{{\mathcal A}}
\def\LC{{\mathcal L}}
\def\CC{{\mathcal C}}  
\def\DC{{\mathcal D}}  
\def\EC{{\mathcal E}}  
\def\QC{{\mathcal Q}}  
\def\PC{{\mathcal P}}  

\DeclareMathOperator{\End}{End}

\DeclareMathOperator{\SL}{SL}
\DeclareMathOperator{\GL}{GL}
\DeclareMathOperator{\grO}{O}

\DeclareMathOperator{\Ball}{\mathbf{B}}

\DeclareMathOperator{\Gr}{Gr}        

\DeclareMathOperator{\Span}{Span}
\DeclareMathOperator{\Supp}{Supp}

\DeclareMathOperator{\indic}{\mathbf{1}}
\DeclareMathOperator{\En}{\omega}        

\DeclareMathOperator{\dang}{d_\measuredangle}        
\DeclareMathOperator{\Vang}{\mathcal{V}_{\!\measuredangle}} 
\DeclareMathOperator{\dimH}{\dim_H}

\newcommand{\transp}[1]{\prescript{t}{}{#1}} 
\newcommand{\abs}[1]{\lvert#1\rvert}    
\newcommand{\norm}[1]{\lVert#1\rVert}   

\newcommand{\ceil}[1]{\left\lceil#1\right\rceil}   

\newcommand{\inv}[1]{#1^{-1}}                  
\newcommand{\Inv}[1]{\frac{1}{#1}}               

\newcommand{\ens}[1]{\left\lbrace#1\right\rbrace}     
\newcommand{\ensA}[1]{\{1,\dotsc,#1\}}     

\newcommand{\Prob}[2][]{\ifthenelse{\equal{#1}{}}   
            {\mathbb{P}}               
            {\mathbb{P}_{#1}}\bigl[#2\bigr]}

\newcommand{\Espr}[2][]{\ifthenelse{\equal{#1}{}}   
            {\mathbb{E}}               
            {\mathbb{E}_{#1}}\bigl[#2\bigr]}

\newcommand{\muex}[1][]{\ifthenelse{\isempty{#1}}   
            {\mu_{\mathrm{ex}}}               
            {\mu_{\mathrm{ex}}^{#1}}}

\newcommand{\Ncov}[1][\delta]{\mathcal{N}_{#1}}

\DeclareMathOperator{\ECreg}{\EC_{reg}}

\renewcommand{\bullet}{\boldsymbol{\,\cdot\,}}

\title{Orthogonal projections of discretized sets}

\author{Weikun He}
\date{\today}
\address{Laboratoire de Mathématiques d'Orsay, Univ. Paris-Sud, Université Paris-Saclay, 91405 Orsay, France.}
\address{Einstein Institute of Mathematics, The Hebrew University of Jerusalem, Jerusalem 91904, Israel.}
\email{weikun.he@mail.huji.ac.il}

\begin{document}

\begin{abstract}
We generalize Bourgain's discretized projection theorem to higher rank situations. Like Bourgain's theorem, our result yields an estimate for the Hausdorff dimension of the exceptional sets in projection theorems formulated in terms of Hausdorff dimensions. This estimate complements earlier results of Mattila and Falconer.
\end{abstract}

\maketitle

\input{proj}

\input{prelim}

\input{prepa}

\input{proof}

\input{hausd}

\bibliographystyle{abbrv} 
\bibliography{proj}

\end{document}

%% file: proj.tex
\section{Introduction}
Fractal properties of orthogonal projections of subsets the Euclidean space have been intensively studied in fractal geometry (See the survey \cite{FalconerFraserJin} for history and recent development). One of the fundamental problems asks for lower bounds on the size of the projections of a given set to different directions. Since, in general, we do not expect the projection to be large in every direction, we ask more precisely to bound from above the size of the set of exceptional directions where an exceptional direction means a subspace onto which the projection is small. In this problem, the notion of size varies according to the context. For example, in a fractal geometric context, it is often the Lebesgue measure or the Hausdorff dimension. In a discretized setting, we measure the size of a set by its covering number by $\delta$-balls where $\delta > 0$ is the observing scale. In this setting, Bourgain established a discretized projection theorem~\cite[Theorem 5]{Bourgain2010} concerning rank one projections. The primary goal of the present paper is to generalize Bourgain's result to higher rank projections.

\subsection{Statement of the main result}
Let $0< m < n$ be positive integers. Let $\delta > 0$. We endow $\R^n$ with its usual Euclidean structure. For $x \in \R^n$, $\Ball(x,\delta)$ stands for the closed ball of radius $\delta$ and center $x$. Let $A$ be a bounded subset of $\R^n$. We write $\Ncov(A)$ for the minimal number of balls of radius $\delta$ that is needed in order to cover $A$. This number represents the size of $A$ at scale $\delta$. 

We denote by $\Gr(\R^n,m)$ the Grassmannian of $m$\dash{}dimensional subspaces in $\R^n$. For $V \in \Gr(\R^n,m)$, $\pi_V \colon \R^n \to V$ stands for the orthogonal projection to $V$. If $W \in \Gr(\R^n,n-m)$, we define
\[\dang(V,W) = \abs{\det(v_1, \dotsc, v_m, w_1, \dotsc, w_{n-m})},\]
where $(v_1,\dotsc, v_m)$ is an orthonormal basis of $V$ and $(w_1,\dotsc, w_{n-m})$ an orthonormal basis of $W$ and the determinant is with respect to any orthonormal basis of $\R^n$. For example $\dang(V,W) = 0$ if and only if $V$ and $W$ have nontrivial intersection. For $\rho \geq 0$, we denote by $\Vang(W,\rho)$ the set of all $V\in \Gr(\R^n,m)$ such that $\dang(V,W) \leq \rho$. Recall that $\Vang(W,0)$ is a submanifold of codimension $1$ in $\Gr(\R^n,m)$ and belongs to the class of algebraic subvarieties known as Schubert cycles (see for example \cite[Chapter 1, \S 5]{GriffithsHarris}).

Our main result is the following.
\begin{thm}\label{thm:proj}
Let $m < n$ be positive integers. Given $0 < \alpha < n$ and $\kappa > 0$, there exists $\epsilon > 0$ such that the following holds for sufficiently small $\delta > 0$. Let $A$ be a subset of $\R^n$ contained in the unit ball $\Ball(0,1)$. Let $\mu$ be a probability measure on $\Gr(\R^n,m)$. Assume that
\begin{equation} \label{eq:sizeA}
\Ncov(A) \geq \delta^{-\alpha + \epsilon};
\end{equation}
\begin{equation} \label{eq:nonconA} 
\forall \rho \geq \delta,\; \forall x\in \R^n,\quad \Ncov(A \cap \Ball(x,\rho)) \leq \delta^{-\epsilon}\rho^\kappa \Ncov(A);
\end{equation}
\begin{equation} \label{eq:nonconmu} 
\forall \rho \geq \delta,\; \forall W \in \Gr(\R^n, n-m),\quad \mu(\Vang(W,\rho)) \leq \delta^{-\epsilon}\rho^\kappa.
\end{equation}
Then there is a set $\DC \subset \Gr(\R^n,m)$ such that $\mu(\DC) \geq 1 - \delta^\epsilon$ and 
\[\Ncov(\pi_V(A')) \geq \delta^{-\frac{m}{n}\alpha - \epsilon}\]
whenever $V \in \DC$ and $A' \subset A$ is a subset such that $\Ncov(A') \geq \delta^\epsilon \Ncov(A)$.
\end{thm}

The $m=1$ case is due to Bourgain~\cite{Bourgain2010}. For $m \geq 2$, our result is new. Hypothesis~\eqref{eq:nonconA} is a Frostmann type non-concentration condition on $A$. Without it we can have example like $A = \Ball(0,\delta^{1-\frac{\alpha}{n}})$, a ball of radius $\delta^{1-\frac{\alpha}{n}}$, whose size is $\Ncov(A) \approx \delta^{-\alpha}$ but whose projection to any $V \in \Gr(\R^n,m)$ is of size
\[\Ncov(\pi_V(A)) \approx \delta^{- \frac{m}{n} \alpha}.\]
Hypothesis~\eqref{eq:nonconmu} is a non-concentration condition on the distribution of the subspace $V$. The set $\Vang(W,\rho)$ can be thought of as a $\rho$-neighborhood of the Schubert cycle $\Vang(W,0)$. For example if $m=1$, $V$ lives in the projective space and \eqref{eq:nonconmu} is asking $\mu$ to be not concentrated around any projective subspace.
Note that the factor $\delta^{-\epsilon}$ in both \eqref{eq:nonconA} and \eqref{eq:nonconmu} means the non-concentration property needs to be satisfied up to scale $\delta^\epsilon$. So the parameter $\kappa$ is about how good the assumptions are and $\epsilon$ is about how much the assumptions can be relaxed and how good the conclusion is.

\subsection{Fractal geometric consequences}

Just like Bourgain's discretized projection theorem can be used to derive a projection theorem in terms of Hausdorff dimension \cite[Theorem 4]{Bourgain2010}, Theorem~\ref{thm:proj} has the following consequence.
\begin{thm}\label{cr:proj}
Let $m < n$ be positive integers. Given $0 < \alpha < n$ and $\kappa > 0$, there is $\epsilon > 0$ such that the following is true.
Let $A \subset \R^n$ is an analytic set of dimension $\dimH(A) =  \alpha$. Then the set of exceptional directions 
\[\bigl\{V \in \Gr(\R^n,m) \mid \dimH(\pi_V(A)) \leq \frac{m}{n}\alpha + \epsilon\bigr\}\]
does not support any nonzero measure $\mu$ on $\Gr(\R^n,m)$ with the following non-concentration property, 
\[\forall \rho > 0,\; \forall W \in \Gr(\R^n,n - m),\quad \mu(\Vang(W,\rho)) \leq \rho^\kappa.\]
\end{thm}

Endow the Grassmanian $\Gr(\R^n,m)$ with a rotation invariant Riemannian metric so that we can talk about Hausdorff dimension of subsets of $\Gr(\R^n,m)$. Theorem~\ref{cr:proj} applied to a Frostman measure supported on the set of exceptional directions, we get
\begin{coro}\label{cr:projFrost}
Let $m < n$ be positive integers. Given $0 < \alpha < n$ and $\kappa > 0$, there is $\epsilon > 0$ such that the following holds.
Let $A \subset \R^n$ be an analytic set of dimension $\dimH(A) = \alpha$. Then
\[
\dimH \bigl\{V \in \Gr(\R^n,m) \mid  \dimH( \pi_V(A)) \leq \frac{m}{n}\alpha + \epsilon \bigr\} \leq m(n-m) - 1 + \kappa.
\]
\end{coro}
Note that $m(n-m)$ is the dimension of $\Gr(\R^n,m)$. As $\kappa \to 0$, we get 
\begin{equation}\label{eq:projFrost}
\dimH \bigl\{V \in \Gr(\R^n,m) \mid  \dimH( \pi_V(A)) \leq \frac{m}{n}\dimH(A) \bigr\} \leq m(n-m) - 1.
\end{equation}
This may be compared to estimates already known.
\begin{thm}[Mattila~\cite{Mattila1975}, Falconer~\cite{Falconer1982}, see also~{\cite[\S 5.3]{Mattila_Fourier}}]
\label{thm:Mattila}
Let $A \subset \R^n$ be an analytic set of Hausdorff dimension $\dimH(A) = \alpha$. For any $0 < s \leq \min\{\alpha,m\}$,
\[
\dimH \{V \in \Gr(\R^n,m) \mid  \dimH( \pi_V(A)) < s\} \leq m(n-m) - (\max\{\alpha,m\}-s);
\]
\end{thm}

Compared to Theorem~\ref{thm:Mattila}, the estimate \eqref{eq:projFrost} provides new information in the following two situations: 
\begin{enumerate}
\item (Projection to lines) $m=1$ and $\dimH(A) \in {]0, 1 + \frac{1}{n-1}[}$,
\item (Projection to hyperplanes) $m = n - 1$ and $\dimH(A) \in {]n- 1 -\frac{1}{n-1}, n[}$.
\end{enumerate}
For example, for $n = 2$ and $m = 1$, the case treated by Bourgain~\cite{Bourgain2010},
\[
\dimH \bigl\{\theta \in \Gr(\R^2,1) \mid  \dimH( \pi_\theta(A)) \leq \frac{1}{2}\dimH(A) \bigr\} = 0,
\]
for all analytic sets $A$ such that $0 < \dimH(A) < 2$. This estimate is also obtained by Oberlin~\cite{Oberlin2012} using different methods. Bourgain's approach has the advantage of giving an estimate with the $\epsilon$ and $\kappa$ terms, or in other words that for any $c > 0$,
\[
\lim_{\epsilon \to 0} \, \sup_A \, \dimH \bigl\{\theta \in \Gr(\R^2,1) \mid  \dimH( \pi_\theta(A)) \leq \frac{1}{2}\dimH(A) + \epsilon \bigr\} = 0,
\]
where $A$ ranges over all analytic sets with Hausdorff dimension between $c$ and $2-c$. Note also that Corollary~\ref{cr:projFrost} can be reformulated in a similar way.

Theorem~\ref{cr:proj} can be combined with Remez-type inequalities to study restricted family of projections. Instead of looking at projections to all subspaces, we restrict our attention to a family of subspaces. The non-concentration property in Theorem~\ref{cr:proj} translates to a transversality condition on the family. In the following corollary, we assume the family to be analytic and not contained in any proper Schubert cycle. 

\begin{coro}\label{cr:projRes}
Let $m < n$ be positive integers. Given $0 < \alpha < n$ and $\kappa > 0$, there is $\epsilon > 0$ such that the following holds.
Let $p \geq 1$ be an integer and $\Omega \subset \R^p$ a connected open subset. Let $V \colon \Omega \to \Gr(\R^n,m)$ be a real analytic map.
Let $A \subset \R^n$ be an analytic set of dimension $\dimH(A) = \alpha$. If for any $W \in \Gr(\R^n, n-m)$, there exists $t \in \Omega$ such that $V(t) \oplus W = \R^n$,
then for any relatively compact subset $\Omega'$ in $\Omega$, there exists a constant $d = d(V,\Omega') > 0$ such that
\begin{equation}\label{eq:projRes}
\dimH \bigl\{t \in \Omega' \mid  \dimH( \pi_{V(t)}(A)) \leq \frac{m}{n}\alpha + \epsilon \bigr\} \leq p - 1 + d\kappa.
\end{equation}
If moreover $V$ is polynomial then $d$ is independent of $\Omega'$ and proportional to the degree.
\end{coro}

The study of restricted family of projections started long ago and saw significant progress recently. We refer the reader to, for example, \cite{PeresSchlag,JJLL,JJK,FasslerOrponen,Orponen2015,Chen2017,KaenmakiOrponenVenieri,OrponenVenieri}. The recent interest is focused on whether for almost all parameters $t$, the dimension of the projection $\dimH( \pi_{V(t)}(A))$ is at least the minimum between $\dimH(A)$, the original dimension, and $m$, the dimension of the subspaces to which we project (see for example~\cite[Conjecture 1.6]{FasslerOrponen}). Corollary~\ref{cr:projRes} deals with a different but parallel question. Here we compare $\dimH( \pi_{V(t)}(A))$ to $\frac{m}{n}\dimH(A) + \epsilon$. Understandably, the exceptional set is much smaller. 

\subsection{Ergodic motivation}
In~\cite{BFLM}, Bourgain, Furman, Lindenstrauss and Mozes used Bourgain's discretized projection theorem together with harmonic analysis to show equidistributions of linear random walks on the torus. Our primary motivation behind Theorem~\ref{thm:proj} resides also in this ergodic problem. In Bourgain-Furman-Lindenstrauss-Mozes theorem, there is technical assumption which is the proximality. While a subgroup $\Gamma \subset \SL_d(\Z)$ acts on the torus, its transpose $\transp{\Gamma}$ acts on Fourier coefficients. Bourgain's discretized projection theorem is used to study large Fourier coefficients under this action. By the theory of random matrix products, if $\Gamma$ is proximal, then large random products in $\Gamma$ behave like rank one projections composed with rotations, if viewed at an appropriate scale. When $\Gamma$ is not proximal, they behave like rank $p$ projections composed with rotations, where $p \geq 2$ is the proximality dimension of the random walk. Thus, we hope Theorem~\ref{thm:proj} will be useful for understanding the non-proximal situation.

\subsection{Strategy of the proof}
Now we describe an outline of the proof of Theorem~\ref{thm:proj}. Fix integers $0 < m < n$ and a real number $0 < \alpha < n$. For $\epsilon > 0$ and bounded subset $A \subset \R^n$ we define the set of exceptional directions to be
\begin{multline}\label{eq:ECAeps}
\EC(A,\epsilon) = \{V \in \Gr(\R^n,m) \mid \exists A'\subset A,\; \Ncov(A') \geq \delta^\epsilon \Ncov(A)\\
\text{ and }\; \Ncov(\pi_V(A')) < \delta^{-\frac{m}{n}\alpha - \epsilon}\}.
\end{multline}
When there is no ambiguity, we omit the variable $\epsilon$ and write simply $\EC(A)$. Our task is to bound $\mu(\EC(A))$ given the distribution $\mu$ of the subspaces. In order to prove Theorem~\ref{thm:proj} which says $\mu(\EC(A)) \leq \delta^\epsilon$ under the assumptions of the theorem, we prove instead that $\mu(\EC(A')) \leq \delta^\epsilon$ for some subset $A'$ of $A$.

\begin{thm}\label{thm:main}
Let $m < n$ be positive integers. Given $0 < \alpha < n$ and $\kappa > 0$, there exists $\epsilon > 0$ such that the following holds for sufficiently small $\delta > 0$. Let $A$ be a subset of $\R^n$ contained in the unit ball $\Ball(0,1)$. Let $\mu$ be a probability measure on $\Gr(\R^n,m)$. Assume \eqref{eq:sizeA}, \eqref{eq:nonconA} and \eqref{eq:nonconmu}, then there exists $A' \subset A$ such that 
\[\mu(\EC(A')) \leq \delta^\epsilon.\]
\end{thm}
This statement is seemingly weaker, but there is actually a rather formal argument which allows to deduce Theorem~\ref{thm:proj} from Theorem~\ref{thm:main}. We will show this implication in Proposition~\ref{pr:unionAps}.

The proof of Theorem~\ref{thm:main} starts with the special case where $n = 2m$.
\begin{prop}\label{pr:n2mcase}
Theorem~\ref{thm:main} is true if $n = 2m$.
\end{prop}
As in the $m=1$ case in \cite{Bourgain2010}, this special case is proved using a sum-product theorem. For $m > 1$, we need the higher dimensional sum-product estimate established in \cite{He2016} which we recall here. Below and throughout this paper, for subsets $X,Y$ of a linear space, we denote by $X + Y$ their sumset :
\[X + Y = \ens{x + y \mid x \in X,\, y \in Y}.\] 
\begin{thm}[{\cite[Theorem 3]{He2016}}]\label{thm:ActionRn}
Let $m$ be a positive integer. Given $\kappa > 0$ and $\sigma < m$, there is $\epsilon > 0$ such that the following holds for $\delta > 0$ sufficiently small. Let $\AC$ be a subset of the space of linear endomorphisms $\End(\R^m)$ and $X$ a subset of $\R^m$, assume that
\begin{enumerate}
\item \label{it:AinBall} $\AC \subset \Ball(0,\delta^{-\epsilon})$,
\item \label{it:sizeAatrho} $\forall \rho \geq \delta$, $\Ncov[\rho](\AC) \geq \delta^\epsilon \rho^{-\kappa}$,
\item \label{it:irreducible} for any nonzero proper linear subspace $W \subset \R^n$, there is $a \in \AC$ and $w \in W \cap \Ball(0,1)$ such that $d(aw,W) \geq \delta^\epsilon$.
\item \label{it:XinBall} $X \subset \Ball(0,\delta^{-\epsilon})$,
\item \label{it:sizeXatrho} $\forall \rho \geq \delta$, $\Ncov[\rho](X) \geq \delta^\epsilon \rho^{-\kappa}$,
\item \label{it:sizeXleq} $\Ncov(X) \leq \delta^{-\sigma - \epsilon}$.
\end{enumerate}
Then, $\Ncov(X + X) + \max_{a \in \AC}\Ncov(X+aX) \geq \delta^{-\epsilon}\Ncov(X)$.
\end{thm}

The proof of Proposition~\ref{pr:n2mcase} follows closely that in \cite{Bourgain2010}. The main idea is to use additive combinatorial tools such as the Balog-Szemerédi-Gowers theorem to reduce to the situation where $A$ is a cartesian product $X \times X$ with $X \subset \R^m$. Then projections of $X \times X$ to subspaces of dimension $m$ correspond exactly to the sum-product operations $X + aX$, $a \in \End(\R^m)$, in Theorem~\ref{thm:ActionRn}. Finally, Theorem~\ref{thm:ActionRn} shows that the projection gained a factor $\delta^{-\epsilon}$ in size compared to $X$ which has half the dimension of $A$. A technical point appearing in this proof is that the set $X$, which is roughly a projection of $A$, has to satisfy the non-concentration property require by Theorem~\ref{thm:ActionRn}. This is addressed in Lemma~\ref{lm:NCproj}.


Once we have Proposition~\ref{pr:n2mcase} we would like to reduce other cases to it. First, using a simple induction, we show that Theorem~\ref{thm:main} holds if $m$ divides $n$.
\begin{prop}\label{pr:nqmcase}
Let $q \geq 3$ be an integer. If Theorem~\ref{thm:main} is true for $n' = (q - 1)m$ and $m$ then it is also true for $n = qm$ and $m$.
\end{prop}

The proof of Proposition~\ref{pr:nqmcase} goes roughly as follows. If Theorem~\ref{thm:main} fails for $n = qm$ and $m$ with the set $A$. Then for a lot of $V \in \Gr(\R^n,m)$, the projection $\pi_V(A)$ is small : $\Ncov(\pi_V(A)) \leq \delta^{-\frac{m}{n}\alpha-\epsilon}$. This implies that the $\delta$-neighborhood of a fiber of $\pi_V$ has a large intersection with $A$. This means that there is a $n'$\dash{}dimensional slice (of thickness $\delta$) of $A$ which has a covering number $\geq \delta^{-\frac{n'}{n}\alpha + \epsilon}$. Now we can apply Theorem~\ref{thm:main} with $n'$ and $m$ to this slice. The main technical issue appearing here is to ensure that the slice has the correct non-concentration property and this is addressed in Lemma~\ref{lm:NCslice}.

If $m$ does not divide $n$ and $m < \frac{n}{2}$, write $n = qm + r$ with $0< r < m$. We can reduce the $(n,m)$-case to the $(n,qm)$-case. 
\begin{prop}\label{pr:nqmrcase}
Let $0<m<n$ be such that $qm < n$ where $q \geq 1$. If Theorem~\ref{thm:main} is true for $n$ and $m'= qm$ then it is also true for $n$ and $m$.
\end{prop}
The idea is the following. Let $V_1,\dotsc,V_q$ be random $m$-planes distributed independently according to $\mu$. Thanks to the non-concentration property of $\mu$, the sum $V = V_1 + \dotsb + V_q$ is a direct sum in well-spaced position with large probablity. Thus the size of the projection $\pi_V(A)$ is comparable to the product of the sizes of $\pi_{V_i}(A)$, $i \in \ensA{q}$. Applying Theorem~\ref{thm:main} with $n$ and $m' = qm$ to $A$ and the distribution of $V$, we conclude that with large probability, $\pi_V(A)$ has size larger than $\delta^{-\frac{qm}{n}\alpha-q\epsilon}$ and hence for some $i$, $\pi_{V_i}(A)$ has size larger than $\delta^{-\frac{m}{n}\alpha-\epsilon}$.

If $m$ does not divide $n$ and $m > \frac{n}{2}$, write $n = q(n-m) + r$ with $0< r \leq n-m$ and we reduce to the $(n,r)$-case.
\begin{prop}\label{pr:nqn-mrcase}
Let $0<m<n$ be such that $n = q(n - m) +r$ where $q \geq 1$ and $0 < r \leq n-m$. If Theorem~\ref{thm:main} is true for $n$ and $m'= r$ then it is also true for $n$ and $m$.
\end{prop}
This last reduction is the trickiest one. We are in a dual situation to the previous one. Again let $V_1,\dotsc,V_q$ be random $m$-planes distributed independently according to $\mu$. This time we consider the intersection instead of the sum of these subspaces. With large probability, the intersection $V = V_1 \cap \dotsb \cap V_q$ has dimension $r$. Thus, we can apply Theorem~\ref{thm:main} with $n$ and $m' = r$ to $\pi_V(A)$. Then the main task is to relate the size of $\pi_V(A)$ to those of $\pi_{V_i}(A)$. We would like to say that $\pi_V(A)$ being large implies one of the $\pi_{V_i}(A)$ must be large as well. However, this is not true in general. It becomes true only if we know that no fiber of $\pi_V$ has large intersection with $A$ (larger than $\delta^{-\frac{n-r}{n}\alpha-\epsilon}$). This relation is proved in Proposition~\ref{pr:SliceEndelta} using a refinement (Lemma~\ref{lm:EnergyProj}) of a combinatorial projection theorem due to Bollob\'{a}s and Thomason~\cite{BollobasThomason}. It remains to treat the case where there is a fiber of $\pi_V$ having large intersection with $A$ or, in other words, the case where $A$ has a $(n-r)$\dash{}dimensional slice with covering number $\geq \delta^{-\frac{n-r}{n}\alpha-\epsilon}$. The idea is to apply a projection theorem to this slice. Since it has a very large size, we achieve this even without a non-concentration property (Proposition~\ref{pr:Nincr}).

Now let us see how to prove Theorem~\ref{thm:main} by putting these propositions together.
\begin{proof}[Proof of Theorem~\ref{thm:main}]
Propositions~\ref{pr:n2mcase} and \ref{pr:nqmcase} imply the theorem for all pairs $(n,m)$ such that $m$ divides $n$. Consider the following order on pairs of positive integers of the form $(n,m)$, $0 < m <n$. We say $(n,m) \prec (n',m')$ if $(n, \min(m,n-m),m)$ is smaller than $(n',\min(m',n'-m'),m')$ for the lexicographical order.

If the theorem were false then let $(n,m)$ be a $\prec$-minimal pair for which the theorem fails. We know that $m$ does not divide $n$. If $m < \frac{n}{2}$ then write $n = qm + r$ with $0< r < m$. We have $(n,qm) \prec (n,m)$. Hence Proposition~\ref{pr:nqmrcase} contradicts the minimality of $(n,m)$. Otherwise $m > \frac{n}{2}$, then write $n = q(n-m) + r$ with $0< r \leq n-m$. We have $(n,r) \prec (n,m)$ and then Proposition~\ref{pr:nqn-mrcase} contradicts the minimality of $(n,m)$.
\end{proof}

\subsection*{Acknowledgements}
This work is part of my PhD thesis conducted under the supervision of Emmanuel Breuillard and Péter Varj\'u. I am greatly indebted to my advisors for their help. I am also grateful to Nicolas de Saxcé for stimulating conversations and to Julien Barral, Yichao Huang and Elon Lindenstrauss for helpful comments.

%% file: prelim.tex
\section{Preliminaries}
In this section we introduce notation that will be used throughout the paper, then provide some elementary estimates about the Grassmannian and finally recall some tools from additive combinatorics.
\subsection{Notation and basic definitions}
Throughout this paper, $m$ and $n$ will be positive integers that denote dimensions. For any finite set $A$, we denote by $\abs{A}$ its cardinality. We endow $\R^n$ with its usual Euclidean structure. We denote by $\grO(n)$ the orthogonal group on $\R^n$, by $\lambda$ the Lebesgue measure on $\R^n$ and by $\Gr(\R^n,m)$ the Grassmannian of $m$\dash{}dimensional subspaces of $\R^n$. For a linear subspace $V \subset \R^n$, denote by $\pi_V$ the orthogonal projection onto $V$. Recall that there is a unique Euclidean structure on each of the exterior powers $\bigwedge^m \R^n$ for which the standard basis is a orthonormal basis.

Let $\delta > 0$ be a real number that we will refer to as the scale. For a point $x \in \R^n$, we write $\Ball(x,\delta)$ or $x^{(\delta)}$ to denote the closed ball of radius $\delta$ centered at $x$. Let $A$ be a bounded subset of $\R^n$. We denote by $A^{(\delta)}$ the closed $\delta$-neighborhood of $A$. 

When we observe a set $A$ at scale $\delta$, there are several quantities describing the size of $A$. They differ one from another at most by a constant factor depending only on $n$. The first one is the external covering number by $\delta$-balls (also known as the metric entropy), denoted by $\Ncov(A)$. It is defined as the minimal number of points $x_1, \dotsc,x_N$ such that the balls $x_1^{(\delta)}, \dotsc, x_N^{(\delta)}$ cover $A$. Let $\tilde A$ be a maximal $2\delta$-separated subset of $A$. Its cardinality also reflects the size of $A$ at scale $\delta$. We can also consider the Lebesgue measure $\lambda(A^{(\delta)})$ of the $\delta$-neighborhood of $A$. Here is a relation between these quantities.
\begin{lemm}\label{lm:coverLeb}
Let $\delta > 0$ and let $A$ be a bounded subset of $\R^n$. Let $\tilde A$ be a maximal $2\delta$\dash{}separated subset of $A$. Then
\begin{equation}\label{eq:coverLeb}
\Ncov[2\delta](A) \leq \abs{\tilde A} \leq \Ncov(A) \leq \Ncov[1](\Ball(0,2))\, \Ncov[2\delta](A),
\end{equation}
and
\begin{equation*}
\abs{\tilde A} \leq \frac{\lambda(A^{(\delta)})}{\lambda(\Ball(0,\delta))} \leq 2^n \Ncov(A).
\end{equation*}
As a consequence, $\Ncov(A^{(\delta)}) \ll_n \Ncov(A)$.
\end{lemm}

It is sometimes useful to change scale. Clearly, $\Ncov(A)$ is nonincreasing in $\delta$. Conversely, for all $\delta' \geq \delta$, we have
\begin{equation}\label{eq:deltaPrim}
\Ncov(A) \ll_n \Bigl(\frac{\delta'}{\delta}\Bigr)^n\Ncov[\delta'](A).
\end{equation}

If $f \colon \R^m \to \R^n$ is a linear map with $\norm{f} \leq K$ where $K \geq 1$, or more generally if $f : A \to \R^n$ is $K$-Lipschitz, we have
\begin{equation}\label{eq:sizefA}
\Ncov(fA) \ll_n K^n\Ncov(A).
\end{equation}

When we want intersect two discretized sets $A,B \subset \R^n$, we shall take the $\delta$-neighborhood of at least one of the sets before intersecting. Note that $\Ncov(A^{(\delta)}\cap B^{(\delta)})$ can be large while at the same time $A\cap B$ is empty. The same goes with $A^{(2\delta)}\cap B^{(2\delta)}$ and $A^{(\delta)}\cap B^{(\delta)}$. However, we have
\begin{equation}\label{eq:deltacap}
\Ncov(A^{(2\delta)}\cap B) \ll_n \Ncov(A^{(\delta)}\cap B^{(\delta)}) \ll_n \Ncov(A\cap B^{(2\delta)}).
\end{equation}
\subsection{Distance on the Grassmannian}
For linear subspaces $V,W$ of $\R^n$, we define
\[\dang(V,W) = \norm{v_1 \wedge \dotsb \wedge v_r \wedge w_1 \wedge \dotsb \wedge w_s}\]
where $(v_1,\dotsc, v_r)$ is an orthonormal basis of $V$ and $(w_1,\dotsc, w_s)$ an orthonormal basis of $W$. It is a distance when restricted to the projective space $\Gr(\R^n,1)$ but only in this case. For example, $\dang(V,W) = 0$ if and only if $V$ and $W$ have nontrivial intersection and $\dang(V,W) = 1$ if and only if they are orthogonal to each other. For other cases, $\dang(V,W)$ falls between $0$ and $1$.

If $v_1,\dotsc,v_r$ are vectors and $\mathbf{w} = w_1 \wedge \dotsb \wedge w_s$ the wedge product of an orthonormal basis of $W$, then
\begin{equation}\label{eq:wedgePi}
\norm{v_1 \wedge \dotsb \wedge v_r \wedge \mathbf{w}} = \norm{\pi_{W^\perp}(v_1) \wedge \dotsb \wedge \pi_{W^\perp}(v_r)}.
\end{equation}
In particular, if $(v_1,\dotsc,v_r)$ is an orthonormal basis of $V$, then
\begin{equation}\label{eq:wedgePi2}
\dang(V,W) = \norm{\pi_{W^\perp}(v_1) \wedge \dotsb \wedge \pi_{W^\perp}(v_r)}.
\end{equation}

If $f \colon V \to W$ is a linear map between euclidean spaces of same dimension, then the determinant of its matrix expressed in orthonormal bases up to a sign does not depend on the choice of the bases. Moreover, we have
\[\abs{\det(f)} = \norm{f(v_1)\wedge \dotsb \wedge f(v_r)}\] 
where $(v_1,\dotsc,v_r)$ is an orthonormal basis of $V$. Together with \eqref{eq:wedgePi2} this gives yet another definition of $\dang(V,W)$ if $\dim(V) + \dim(W) = n$,
\begin{equation}\label{eq:det=dang}
\dang(V,W) = \abs{\det(\pi_{W^\perp \mid V})},
\end{equation}
where $\pi_{W^\perp \mid V} \colon V \to W^\perp$ denotes the restriction of $\pi_{W^\perp}$ to $V$.

The natural action of the orthogonal group $\grO(n)$ on the Grassmannian preserves $\dang$, i.e.
\[\forall g \in \grO(n),\quad \dang(gV,gW) = \dang(V,W).\]
Consequently if $\dim V + \dim W = n$ then
\begin{equation}\label{eq:perpperp}
\dang(V^\perp,W^\perp) = \dang(V, W),
\end{equation}
because in this case we can always send $V$ to $W^\perp$ (hence $W$ to $V^\perp$) by an element of $\grO(n)$.

Moreover, when we have several subspaces, $V_1,V_2, \dotsc , V_q$ of $\R^n$, we define
\[\dang(V_1, \dotsc , V_q) = \norm{\mathbf{v}_1 \wedge \dotsb \wedge \mathbf{v}_q}\]
where for each $i = 1, \dotsc , q$, $\mathbf{v}_i$ is the wedge product of the elements of an orthonormal basis of $V_i$. For example, if $x_1, \dotsc, x_n \in \R^n$ are unit vectors, then
\[\dang(\R x_1, \dotsc ,\R x_n) = \abs{\det(x_1, \dotsc x_n)}.\]

Obviously, $\dang(V_1,\dotsc,V_q)$ is symmetric in the variables $V_1,\dotsc,V_q$. Below are some other elementary properties of $\dang$.

\begin{lemm}\label{lm:dangUVW}
If $U,V,W$ are linear subspaces of $\R^n$, then
\begin{equation}\label{eq:dangUVW}
\dang(U,V,W) = \dang(U+V,W) \dang(U,V).
\end{equation}
Consequently, if $V_1, \dotsc , V_q$ are also linear subspaces, then
\begin{equation}\label{eq:dangExpand}
\dang(V_1, \dotsc , V_q) = \dang(V_2, V_1) \dang(V_3, V_1 + V_2) \dotsm \dang(V_q,V_1 + \dotsb + V_{q-1});
\end{equation}
\begin{equation}\label{eq:dangExpand2}
\dang(V_1 + \dotsb + V_q,W) \geq \dang(V_1, W) \dang(V_2, V_1 + W) \dotsm \dang(V_q,V_1 + \dotsb + V_{q-1} + W).
\end{equation}
\end{lemm}

\begin{proof}
If the sum $U+V$ is not a direct sum, then $\dang(U,V,W) = 0$ and $\dang(U,V) = 0$. Otherwise, let $\mathbf{u}$ and $\mathbf{v}$ be wedge products of orthonormal bases of $U$ and $V$ respectively. Then $\mathbf{u}\wedge \mathbf{v}/\norm{\mathbf{u}\wedge \mathbf{v}}$ is the wedge product of an orthonormal basis of $U + V$. Then \eqref{eq:dangUVW} follows immediately from the definition.

The estimates \eqref{eq:dangExpand} can be obtained by a simple induction. The inequality \eqref{eq:dangExpand2} follows from \eqref{eq:dangExpand} since, by \eqref{eq:dangExpand}, the right hand side of \eqref{eq:dangExpand2} is equal to $\dang(V_1,\dotsc,V_q,W)$ which, by \eqref{eq:dangExpand} again, is equal to $\dang(V_1, \dotsc , V_q) \dang(V_1 + \dotsb + V_q,W)$. 
\end{proof}

\begin{lemm}\label{lm:Cylinders}
Let $q \geq 2$. Let $V_1,\dotsc,V_q$ be linear subspaces of $\R^n$. If $z \in V_1 + \dotsb + V_q$ then
\begin{equation}
\norm{z} \dang(V_1, \dotsc, V_q) \leq \norm{\pi_{V_1}(z)} + \norm{\pi_{V_2}(z)} + \dotsb + \norm{\pi_{V_q}(z)}
\end{equation}
\end{lemm}

\begin{proof}
We will proceed by induction. Let $q = 2$. Obviously, there is nothing to prove if $V_1 + V_2$ is not a direct sum. Moreover, without loss of generality, we can assume that $\R^n = V_1 + V_2$. Hence also $\R^n = V_1^\perp + V_2^\perp$. Write $z = z_1 + z_2$ with $z_1 \in V_1^\perp$ and $z_2 \in V_2^\perp$. Then by \eqref{eq:wedgePi2},
\[\norm{\pi_{V_1}(z)} = \norm{\pi_{V_1}(z_2)} = \norm{z_2} \dang(V_1^\perp,\R z_2) \geq \norm{z_2} \dang(V_1^\perp,V_2^\perp) = \norm{z_2} \dang(V_1,V_2).\]
Similarly, $\norm{\pi_{V_2}(z)} \geq \norm{z_1} \dang(V_1,V_2)$. We get the lemma for $q=2$ using the triangular inequality.

Now, suppose the lemma is true for some $q \geq 2$. Let us show the lemma for $q + 1$. Let $V'_q = V_q + V_{q+1}$ and $z' = \pi_{V'_q}(z)$. The induction hypothesis applied to $z$ and $(V_1,\dotsc,V_{q-1},V'_q)$ gives
\[\norm{z} \dang(V_1, \dotsc,V_{q-1},V_q + V_{q+1}) \leq \norm{\pi_{V_1}(z)} + \dotsb + \norm{\pi_{V_{q-1}}(z)} +\norm{z'}.\]
The $q=2$ case applied to $z'$ and $(V_q,V_{q+1})$ gives
\[\norm{z'} \dang(V_q,V_{q+1}) \leq \norm{\pi_{V_q}(z')} + \norm{\pi_{V_{q+1}}(z')} = \norm{\pi_{V_q}(z)} + \norm{\pi_{V_{q+1}}(z)}.\]
Recall that $\dang(V_1, \dotsc, V_{q+1}) = \dang(V_1, \dotsc,V_{q-1},V_q + V_{q+1})\dang(V_q,V_{q+1})$. We obtain the desired estimate by multiplying the first inequality by $\dang(V_q,V_{q+1})$ and combining it with the second.
\end{proof}

\begin{lemm}\label{lm:tiltCart}
If $\R^n$ is a direct sum of $V_1,\dotsc,V_q$ then for any bounded subset $A \subset \R^n$,
\begin{equation}\label{eq:tiltCart}
\Ncov(A) \ll_n \dang(V_1,\dotsc,V_q)^{-n} \prod_{i=1}^q \Ncov(\pi_{V_i}(A)).
\end{equation}
\end{lemm}

\begin{proof}
Suppose for each $i \in \ensA{q}$, $\pi_{V_i}(A)$ is covered by the balls $x_i^{(\delta)}$, $x_i \in X_i \subset V_i$. For each $(x_i)_i \in X_1 \times \dotsb \times X_q$, there is a unique $x \in \R^n$ such that $\forall i,\; \pi_{V_i}(x) = x_i$. By Lemma~\ref{lm:Cylinders}, we have
\[\pi_{V_1}^{-1}(x_1^{(\delta)}) \cap \dotsb \cap \pi_{V_q}^{-1}(x_q^{(\delta)}) \subset x^{(\delta')},\]
where $\delta'= \inv{\dang(V_1,\dotsc,V_q)}q\delta$. So $A$ is covered by the balls centered at such $x$. Hence $\Ncov[\delta'](A) \leq \abs{X_1} \dotsm \abs{X_q}$. We then conclude by using the scale change estimate \eqref{eq:deltaPrim}.
\end{proof}

\begin{lemm}\label{lm:V2WandU}
Let $V,W,U$ be linear subspaces of $\R^n$, with $U \subset W$. We have 
\begin{equation}\label{eq:V2WandU}
\dang(V,U+W^\perp) = \dang(V,W^\perp)\dang(\pi_W(V),U).
\end{equation}
\end{lemm}

\begin{proof}
Both sides of \eqref{eq:V2WandU} vanish if the dimension of $V' = \pi_W(V)$ is smaller than that of $V$. So we can assume that $\dim V' = \dim V =r$. Let $(v_1,\dotsc,v_r)$ be an orthonormal basis of $V$. Then $(\pi_W(v_1),\dotsc,\pi_W(v_r))$ is a basis of $V'$. Moreover, by~\eqref{eq:wedgePi}, we have
\[\norm{\pi_W(v_1)\wedge \dotsb \wedge \pi_W(v_r)} = \dang(V,W^\perp)\]
and
\[\norm{\pi_W(v_1)\wedge \dotsb \wedge \pi_W(v_r)\wedge \mathbf{u}} = \dang(V,U,W^\perp),\]
where $\mathbf{u}$ is the wedge product an orthonormal basis of $U$. The desired equality \eqref{eq:V2WandU} follows from the fact 
\[\dang(V',U) = \frac{\norm{\pi_W(v_1)\wedge \dotsb \wedge \pi_W(v_r)\wedge \mathbf{u}}}{\norm{\pi_W(v_1)\wedge \dotsb \wedge \pi_W(v_r)}}\]
and Lemma~\ref{lm:dangUVW} applied to $V,U,W^\perp$ :
\[\dang(V,U,W^\perp) = \dang(U,W^\perp) \dang(V,U+W^\perp) = \dang(V,U+W^\perp). \qedhere\]
\end{proof}

\begin{lemm}
Let $V,W$ be linear subspaces of $\R^n$. If $V' = \pi_W(V)$, then for all $x \in W$,
\begin{equation}\label{eq:Wproj2V}
\dang(V,W^\perp) \norm{\pi_{V'}(x)} \leq \norm{\pi_V(x)} \leq \norm{\pi_{V'}(x)}.
\end{equation}
\end{lemm}

\begin{proof}
Since $V' = \pi_W(V)$, we have ${V'}^\perp \cap W \subset V^\perp$. Hence we can write $x = y + z$ with $y = \pi_{V'}(x) \in V'$ and $z \in {V'}^\perp \cap W \subset V^\perp$. Then $\pi_V(x) = \pi_V(y)$. This gives the second inequality in \eqref{eq:Wproj2V}.

It is clear that $V$ and $V'$ have different dimensions if and only if $V$ and $W^\perp$ have nontrivial intersection, which is equivalent to $\dang(V,W^\perp)= 0$. In this case, the first inequality in the lemma holds. 

Thus we can assume $\dim V = \dim V'$. By \eqref{eq:wedgePi2}, $\norm{\pi_V(y)} = \dang(\R y, V^\perp) \norm{y}$. We know that
\[\dang(\R y, V^\perp) \geq \dang(V',V^\perp)= \dang(V,{V'}^\perp)\]
by the fact that $\R y \subset V'$ and \eqref{eq:perpperp}. Observe that $V'^\perp = {V'}^\perp\cap W + W^\perp$. Also, 
\[\dang(V, {V'}^\perp\cap W + W^\perp) = \dang(V,W^\perp) \dang(V',{V'}^\perp\cap W) = \dang(V,W^\perp),\]
by Lemma~\ref{lm:V2WandU} applied to $V,W$ and $U= {V'}^\perp\cap W$. Hence $\norm{\pi_V(y)} \geq \dang(V,W^\perp) \norm{y}$,
which proves the first inequality in \eqref{eq:Wproj2V}.
\end{proof}

\begin{lemm}\label{lm:fromVp2V}
Let $V,W$ be linear subspaces of $\R^n$ such that $\dang(V,W^\perp) > 0$. Write $V' = \pi_W(V)$. For any bounded subset $A \subset W$,
\begin{equation}\label{eq:fromVp2V}
\Ncov(\pi_{V'}(A)) \ll_n \dang(V,W^\perp)^{-n} \Ncov(\pi_V(A)).
\end{equation}
In particular, if moreover $\dim V = \dim W$, then for any bounded subset $A \subset W$,
\begin{equation}\label{eq:fromW2V}
\Ncov(A) \ll_n \dang(V,W^\perp)^{-n} \Ncov(\pi_V(A)).
\end{equation}
\end{lemm}

\begin{proof}
Since $\dang(V,W^\perp) > 0$, $\pi_V$ restricted to $W$ is surjective. Hence we can cover $\pi_V(A)$ by the balls $\pi_V(b)^{(\delta)}$, $b \in \tilde A$ for some $\tilde A \subset W$ with $\abs{\tilde A} = \Ncov(\pi_V(A))$. Then $\pi_{V'}(A)$ is covered by the balls $\pi_{V'}(b)^{(\delta')}, b \in \tilde A$ with $\delta' = \inv{\dang(V,W^\perp)}\delta$. Indeed, $\forall a \in A$, there is $b \in \tilde A$ such that $\norm{\pi_V(a -b)} \leq \delta$. Hence, by \eqref{eq:Wproj2V}, $\norm{\pi_{V'}(a - b)} \leq \delta'$. Thus $\Ncov[\delta'](\pi_{V'}(A)) \leq \Ncov(\pi_V(A))$, which yields \eqref{eq:fromVp2V} using \eqref{eq:deltaPrim}.
\end{proof}

\subsection{Intersections}
Here we collect two useful lemmata about intersections and unions of intersections.

The first one is about intersections of large subsets. Let $A$ be a Borel set in $\R^n$. Let $\Theta$ be an index set equipped with a probability measure $\mu$ and for each $\theta \in \Theta$, we have a Borel subset $A_\theta$ of $A$. We need appropriate measurability, namely, the map $(x,\theta) \mapsto \indic_{A_\theta}(x)$ is required to be measurable.
\begin{lemm} \label{lm:bigCap}
In the situation described above, if there is $K \geq 1$ such that $\forall \theta \in \Theta$, $\lambda(A_\theta) \geq \lambda(A)/K$, then for any positive integer $q > 0$,
\[\mu^{\otimes q}(\bigl\{(\theta_1,\dotsc,\theta_q) \mid \lambda(A_{\theta_1}\cap \dotsb \cap A_{\theta_q}) \geq \frac{\lambda(A)}{2K^q}\bigr\}) \geq \frac{1}{2K^q}.\]
\end{lemm}

\begin{proof}
By Fubini's theorem and then Jensen's inequality,
\begin{align*}
 &\int \lambda(A_{\theta_1}\cap \dotsb \cap A_{\theta_q}) \dd\mu^{\otimes q}(\theta_1,\dotsc,\theta_q)\\
=&\int_A \int \indic_{A_{\theta_1}}(x) \dotsm  \indic_{A_{\theta_q}}(x) \dd\mu^{\otimes q}(\theta_1,\dotsc,\theta_q) \dd\lambda(x)\\
=& \lambda(A)\int_A \bigl(\int \indic_{A_\theta}(x) \dd\mu(\theta)\bigr)^q \frac{\dd\lambda(x)}{\lambda(A)}\\
\geq& \lambda(A) \bigl(\int_A \int \indic_{A_\theta}(x) \dd\mu(\theta) \frac{\dd\lambda(x)}{\lambda(A)} \bigr)^q\\
=& \lambda(A) \bigl(\int \frac{\lambda(A_\theta)}{\lambda(A)} \dd\mu(\theta) \bigr)^q\\
\geq& \frac{\lambda(A)}{K^q}.
\end{align*}
The lemma follows.
\end{proof}

The next lemma is about small probability events happening simultaneously. Let $(E,\mu)$ be a probability space. Suppose we have a collection of subsets $(E_i)_{i \in \ensA{N}}$ of $E$. We will think $E_i$ as events with small probability and we want to estimate the probability such that a lot of them happen together. Here "a lot" is relatively to weights we give to the events. Let $(a_i)_{i \in \ensA{N}}$ be non-negative real numbers such that $\sum_{i=1}^N a_i = 1$. For $I \subset \ensA{N}$, write $a_I = \sum_{i \in I}a_i$. The following lemma is an easy consequence of Markov's inequality. 
\begin{lemm} \label{lm:smallCap}
With the notation above, we have, for any $a > 0$,
\[\mu\bigl( \bigcup_{I \mid a_I \geq a} \bigl( \bigcap_{i \in I}E_i\bigr)\bigr) \leq \inv{a}\max_{i\in \ensA{N}}\mu(E_i).\]
\end{lemm}

\begin{proof}
Consider the Bernoulli random variables $X_i = \indic_{E_i}$ for $i = 1,\dotsc,N$ so that $\mu(E_i) = \Espr{X_i}$ and 
\[\mu\bigl( \bigcup_{I \mid a_I \geq a} \bigcap_{i \in I}E_i\bigr) = \Prob{\sum_{i=1}^N a_iX_i \geq a}.\]
Then it follows from Markov's inequality that
\[\Prob{\sum_{i=1}^N a_iX_i \geq a} \leq \inv{a}\Espr{\sum_{i=1}^N a_iX_i} \leq \inv{a}\max_{i\in \ensA{N}}\Espr{X_i}.\]
This finishes the proof.
\end{proof}

\subsection{Additive combinatorial tools}
Let $A,B,C$ be bounded subsets of $\R^n$. We look at them at scale $\delta > 0$. We will use several well-known results from additive combinatorics in our metric entropic setting. We shall use some usual notation from additive combinatorics :
\[A+B = \ens{a + b \mid a \in A,\, b \in B},\]
\[A-B = \ens{a - b \mid a \in A,\, b \in B},\]
and for integer $k\geq 1$, $kA$ denotes the $k$-fold sumset $A + \dotsb + A$.
\begin{lemm}[Ruzsa triangular inequality]\label{lm:RuzsaTri}
We have
\[\Ncov(B)\Ncov(A-C) \ll_n \Ncov(A-B)\Ncov(B-C).\]
\end{lemm}

\begin{lemm}[Plünnecke-Ruzsa inequality]\label{lm:RuzsaSum}
For all $K \geq 1$, if $\Ncov(A + B) \leq K \Ncov(B)$ then for all natural number $k$ and $l$,
\[\Ncov(kA - lA) \ll_n K^{k+l} \Ncov(B).\]
\end{lemm}

Both lemmata above can be obtained by approximating $\R^n$ by the lattice $\delta.\Z^d$ and then using its discrete counterpart (see for example \cite{TaoVu}) as a black box. More precisely for a subset $A\subset \R^n$, we define
\[\tilde A = \bigl\{a \in \delta \cdot \Z^n \mid A \cap a^{(n\delta)} \neq \varnothing\bigr\}.\]
Then $A \subset {\tilde A}^{(n\delta)}$ and $\tilde A \subset A^{(n\delta)}$. These inclusions behave nicely under addition and subtraction.

Before stating the Balog-Szemerédi-Gowers theorem in the discretized setting let us recall some basic facts about energy in the discrete setting. Let $\phi \colon X \to Y$ be a map between discrete sets and $A$ a finite subset of $X$, define the $\phi$-energy of $A$ to be
\[\En(\phi,A) = \sum_{y \in Y} \abs{A \cap \inv{\phi}(y)}^2.\]
In other words, it is the square of the $l^2$-norm of the push-forward of the counting measure on $A$ under $\phi$ or the number of collisions of the map $\phi_{\mid A}$ :
\[\En(\phi,A) = \norm{\phi_*\!\indic_A}_2^2 = \#\ens{(a_1,a_2) \in A \times A \colon \phi(a_1) = \phi(a_2)}.\]
For example, the usual additive energy between two subsets $A$ and $B$ in an abelian group $G$ is $\En(+,A\times B)$ where $+ \colon G \times G \to G$ denotes the group law of $G$. 

When nothing is known about $\phi$, $\En(\phi,A)$ can be as small as $\abs{A}$ (when $\phi$ is injective) and as large as $\abs{A}^2$ (when $\phi$ is constant on $A$). If the image of $A$ by $\phi$ is small then the energy is large by the Cauchy-Schwarz inequality :
\begin{equation}\label{eq:phiEnergyD}
\En(\phi,A) \geq \frac{\abs{A}^2}{\abs{\phi(A)}}.
\end{equation}
The converse is not true. Nevertheless, we have a partial converse.
\begin{lemm}\label{lm:EnergySubset}
Suppose there are $K,M > 0$ such that $\En(\phi,A) \geq \frac{M}{K}\abs{A}$ and for all $y \in Y$, $\abs{A \cap \inv{\phi}(y)} \leq M$. Then there exists $A' \subset A$ such that $\abs{A'} \geq \Inv{2K}\abs{A}$ and $\abs{\phi(A')} \leq \frac{2K}{M}\abs{A}$.
\end{lemm}

\begin{proof}
The idea is to trim off small fibers. We consider 
\[Y' = \bigl\{y \in Y \mid \abs{A \cap \inv{\phi}(y)} \geq \frac{M}{2K}\bigr\}\]
and let $A' = \inv{\phi}(Y')$. By the definition $Y'$, we have
\[\abs{A} \geq \sum_{y \in Y'} \abs{A \cap \inv{\phi}(y)}  \geq \frac{M}{2K} \abs{Y'}.\]
Hence $\abs{\phi(A')} \leq \frac{2K}{M}\abs{A}$.

From the definition of the energy,
\begin{align*}
\omega(\phi,A) &\leq \frac{M}{2K}\sum_{y\notin Y} \abs{A \cap \inv{\phi}(y)} + M \sum_{y \in Y'} \abs{A \cap \inv{\phi}(y)}\\
&\leq  \frac{M}{2K}\abs{A} + M\abs{A'}.
\end{align*}
It follows that $\abs{A'} \geq \Inv{2K}\abs{A}$.
\end{proof}

What the Balog-Szemerédi-Gowers theorem roughly says is that if $\phi$ is a group law (or has some injectivity property similar to a group law) and $A$ is a Cartesian product then the conclusion of $A'$ in the conclusion of the lemma can be chosen to be a Cartesian product.

For discretized sets we have an analogous notion of energy. Let $\phi \colon X \to Y$ be a map between metric spaces and $A$ a bounded subset of $X$. We define the $\phi$-energy of $A$ at scale $\delta$ as
\[\En_\delta(\phi,A) = \Ncov\bigl(\ens{(a,a') \in A \times A \mid d(\phi(a),\phi(a')) \leq \delta}\bigr).\]
Here we adhere to the convention that the distance on any Cartesian product $X \times Y$ of metric spaces is such that
\[d\bigl((x,y),\,(x',y')\bigr)^2 = d(x,x')^2 + d(y,y')^2,\]
for all pairs $(x,y),\, (x',y') \in X\times Y$.

The analogue of inequality~\eqref{eq:phiEnergyD} is true. Namely, if $A$ is a bounded subset of $\R^n$ and $\phi$ is defined on $\R^n$ then
\begin{equation}\label{eq:phiEnergy}
\En_\delta(\phi,A) \gg_n \frac{\Ncov(A)^2}{\Ncov(\phi(A))}.
\end{equation}

We also remark that if $\psi \colon A \to \R^n$ is $K$-Lipschitz with $K \geq 1$ and $\phi \colon \R^n \to Y$ is an another map, then it follows from \eqref{eq:sizefA} that 
\begin{equation}\label{eq:Enphipsi}
\En_\delta(\phi, \psi A) \ll_n K^{2n} \En_\delta(\phi \circ \psi, A).
\end{equation}

We will need the following additive version of the Balog-Szemerédi-Gowers theorem which gives a nice criterion for the additive energy between two sets to be large. See for example~\cite[Theorem 6.10]{Tao_productset} where it is proved in a much broader context.
\begin{thm}[Balog-Szemerédi-Gowers theorem]\label{thm:BSG+}
Let $K \geq 1$ be a parameter. Let $A$ and $B$ be bounded subsets of $\R^n$. If
\[\En_\delta(+,A\times B) \geq \frac{1}{K}\Ncov(A)^{\frac{3}{2}}\Ncov(B)^{\frac{3}{2}},\]
then there exists $A' \subset A$ and $B' \subset B$ such that 
\[\Ncov(A') \gg_n K^{-O(1)}\Ncov(A),\; \Ncov(B') \gg_n K^{-O(1)}\Ncov(B)\]
and 
\[\Ncov(A' + B') \ll_n K^{O(1)}\Ncov(A)^{\Inv{2}}\Ncov(B)^{\Inv{2}}.\]
\end{thm}

%% file: prepa.tex
\section{Technical lemmata}
In this section, we show the deduction of Theorem~\ref{thm:proj} from Theorem~\ref{thm:main} and collect several other lemmata which are needed in the next section. Since they are mostly about technical details, it is advisable to skip their proofs for a first reading. In this section, implied constants in Landau notations $O(f)$ and Vinogradov notations $f \ll g$ may depend on the dimension $n$ and the parameter $\kappa$. Every statement is true only for $\delta > 0$ sufficiently small and by sufficiently small we mean smaller that a constant depending on all other parameters (e.g. $n$, $m$, $\alpha$, $\kappa$ and $\epsilon$) but not on $A$ nor on $\mu$. Typically, if $C = O(1)$ then $C \leq \delta^{-\epsilon}$.

\subsection{Proof of Theorem~\ref{thm:proj} admitting Theorem~\ref{thm:main}}
We deduce Theorem~\ref{thm:proj} from Theorem~\ref{thm:main}.

\begin{prop}\label{pr:unionAps}
Assume that $0 < m < n$, $0 < \alpha < n$, $\kappa > 0$ and $\epsilon > 0$ are parameters that make Theorem~\ref{thm:main} true. Let $A$ be a subset of $\R^n$ contained in the unit ball. Let $\mu$ be a probability measure on $\Gr(\R^n,m)$. Assume that $\mu$ satisfies \eqref{eq:nonconmu} and $A$ satisfies
\begin{equation*}
\Ncov(A) \geq \delta^{-\alpha + \frac{\epsilon}{2}};
\end{equation*}
\begin{equation*}
\forall \rho \geq \delta,\; \forall x\in \R^n,\quad \Ncov(A \cap \Ball(x,\rho)) \leq \delta^{-\frac{\epsilon}{2}}\rho^\kappa \Ncov(A).
\end{equation*}
Then
\begin{equation}\label{eq:unionAps}
\mu\bigl(\EC(A,\frac{\epsilon}{3})\bigr) \leq \delta^{\frac{\epsilon}{2}}.
\end{equation}
\end{prop}

The idea is the following. A first application of Theorem~\ref{thm:main} gives a subset $A' \subset A$ with $\mu(\EC(A',\epsilon)) \leq \delta^\epsilon$. Either $A'$ is large enough in which case we are done or we can cut $A'$ out of $A$ and apply Theorem~\ref{thm:main} again. This will give us another subset $A'$. Then we iterate until the union of these $A'$s is large enough.
\begin{proof}
Let $N\geq 0$ be an integer. Suppose we have already constructed $A_1,\dotsc,A_N$ such that $A_i^{(\delta)}$ are pairwise disjoint and $\mu(\EC(A_i,\epsilon)) \leq \delta^\epsilon$ for every $i = 1,\dotsc,N$. Either we have
\begin{equation}\label{eq:AprimeBIG}
\Ncov\bigl(A \setminus \bigcup_{i=1}^N A_i^{(2\delta)}\bigr) \leq \delta^{\frac{\epsilon}{2}}\Ncov(A),
\end{equation}
in which case we stop, or the set $A \setminus \bigcup_{i=1}^N A_i^{(2\delta)}$ satisfies both \eqref{eq:sizeA} and \eqref{eq:nonconA}. In the latter case Theorem~\ref{thm:main} gives us $A_{N+1} \subset A \setminus \bigcup_{i=1}^N A_i^{(2\delta)}$ with $\mu(\EC(A_{N+1},\epsilon)) \leq \delta^\epsilon$. By construction, $A_{N+1}^{(\delta)}$ is disjoint with any of $A_i^{(\delta)}$, $i = 1,\dotsc,N$. 

When this procedure ends write $A_0 = \bigcup_{i=1}^N A_i$. Then \eqref{eq:AprimeBIG} says $\Ncov(A \setminus A_0^{(2\delta)}) \leq \delta^\frac{\epsilon}{2} \Ncov(A)$. Moreover, by the disjointness of $A_1^{(\delta)}, \dotsc , A_N^{(\delta)}$, we have
\[\Ncov(A_0) = \sum_{i=1}^N \Ncov(A_i).\]
Set $a_i = \frac{\Ncov(A_i)}{\Ncov(A_0)}$. We claim that
\[\EC(A,\frac{\epsilon}{3}) \subset \bigcup_I \bigcap_{i \in I} \EC(A_i,\epsilon),\]
where the index set $I$ runs over subsets of $\ensA{N}$ with $\sum_{i \in I}a_i \geq \delta^\frac{\epsilon}{2}$. The desired upper bound \eqref{eq:unionAps} then follows immediately from Lemma~\ref{lm:smallCap}.

We now proceed to show the claim. Let $V \in \EC(A,\frac{\epsilon}{3})$. By definition, there exists $A' \subset A$ with $\Ncov(A') \geq \delta^\frac{\epsilon}{3} \Ncov(A)$ and $\Ncov(\pi_V(A')) \leq \delta^{-\frac{m}{n}\alpha - \frac{\epsilon}{3}}$. Consider the index set $I$ defined as
\[I = \{i \in \ensA{N} \mid \Ncov(A^{\prime(2\delta)}\cap A_i) \geq \delta^\epsilon \Ncov(A_i)\}.\]
We have, by Lemma~\ref{lm:coverLeb} and \eqref{eq:deltacap},
\begin{align*}
\Ncov(A') - \Ncov(A \setminus A_0^{(2\delta)}) &\leq \sum_{i=1}^n \Ncov(A' \cap A_i^{(2\delta)})\\
&\ll \sum_{i \in I} \Ncov(A_i) + \sum_{i \notin I} \Ncov(A^{\prime(2\delta)}\cap A_i)\\
&\ll \sum_{i\in I} a_i\Ncov(A) + \delta^\epsilon \Ncov(A)
\end{align*}
Hence $\sum_{i \in I}a_i \geq \delta^\frac{\epsilon}{2}$. On the other hand, for all $i \in I$, since
\[\Ncov(\pi_V(A^{\prime(2\delta)}\cap A_i)) \leq \Ncov(\pi_V(A')^{(2\delta)}) \ll \Ncov(\pi_V(A')),\]
we have
\[\Ncov(\pi_V(A^{\prime(2\delta)}\cap A_i)) \leq \delta^{-\frac{m}{n}\alpha - \epsilon}.\]
Hence $V \in \EC(A_i,\epsilon)$ for all $i \in I$. This finishes the proof of the claim.
\end{proof}

\subsection{Action of linear transformations}
\label{ss:actGL}
Clearly, all the assumptions and the conclusion of Theorem~\ref{thm:main} are invariant under the action of the orthogonal group $\grO(n)$. The next proposition states that the action of a $\delta^{-\epsilon}$-bi-Lipschitz linear transformation only affects them by a factor of $\delta^{O(\epsilon)}$. Here, while $f \in \GL(\R^n)$ acts on $\R^n$ in the usual way, it acts on the Grassmannian by multiplication by $f^\perp \coloneqq (\inv{f})^*$ or equivalently, $f^\perp V = (f V^\perp)^\perp$ for all $V \in \Gr(\R^n,m)$. 
\begin{lemm}\label{lm:GLaction}
Let $0 < m < n$ be dimensions. Let $\epsilon > 0$. Let $f \in \GL(\R^n)$ with $\norm{f} + \norm{\inv{f}} \leq \delta^{-\epsilon}$. Let $A$ be a bounded subset of $\R^n$ and $\mu$ a probability measure on $\Gr(\R^n,m)$.
\begin{enumerate}
\item For each of the conditions \eqref{eq:sizeA}--\eqref{eq:nonconmu} of Theorem~\ref{thm:proj}, if it holds for $A$ and $\mu$ with the parameters $\alpha$, $\kappa$ and $\epsilon$ then it also holds for the image set $f A$ and the image measure $f^\perp_* \mu$ with the parameters $\alpha$, $\kappa$ and $O(\epsilon)$ in the place of $\epsilon$.
\item \label{it:GLpifA} For all $V \in \Gr(\R^n,m)$, $\Ncov(\pi_{f^\perp V}(fA)) \leq \delta^{-O(\epsilon)}\Ncov(\pi_V(A))$.
\item We have $\mu(\EC(A,\epsilon)) \leq (f^\perp_* \mu)\bigl(\EC(f A ,O(\epsilon))\bigr)$. In particular, if the conclusion of Theorem~\ref{thm:main} holds for $fA$ and $f^\perp_* \mu$ with some $\epsilon' > 0$ in the place of $\epsilon$ then it holds for $A$ and $\mu$ with $\epsilon = \frac{\epsilon'}{O(1)}$.
\item For all $V \in \Gr(\R^n,m)$ and all $x \in f^\perp V$,
\[\Ncov\bigl(fA \cap \pi_{f^\perp V}^{-1}(x^{(\delta)})\bigr) \leq \delta^{-O(\epsilon)} \max_{y \in V} \Ncov\bigl(A \cap \pi_V^{-1}(y^{(\delta)})\bigr).\]
\end{enumerate}
\end{lemm}

\begin{proof}
The statement about the conditions \eqref{eq:sizeA} and \eqref{eq:nonconA} follows immediately from the inequality \eqref{eq:sizefA}. As for the condition~\eqref{eq:nonconmu}, it suffices to prove that for all $W \in \Gr(\R^n,n-m)$ and all $\rho \geq \delta$, 
\begin{equation}\label{eq:GLonGr}
f^\perp_*\mu\bigl(\Vang(f^\perp W,\rho)\bigr) \leq \mu\bigl(\Vang(W,\delta^{-O(\epsilon)}\rho)\bigr).
\end{equation}
From the Cartan decomposition of $f$, we see easily that $\forall r = 1, \dotsc, n$, $\norm{\bigwedge^r f^\perp} + \norm{(\bigwedge^r f^\perp)^{-1}} \leq \delta^{-O(\epsilon)}$. For $V \in \Gr(\R^n,m)$, let $\mathbf{v}$ be the wedge product of an orthonormal basis of $V$ and $\mathbf{w}$ that of $W$. We have
\begin{align*}
\dang(f^\perp V,f^\perp W) &= \frac{\norm{(\bigwedge^n f^\perp)(\mathbf{v}\wedge \mathbf{w})}}{\norm{(\bigwedge^m f^\perp)\mathbf{v}}\,\norm{(\bigwedge^{n-m} f^\perp) \mathbf{w}} }\\
&\geq \frac{\norm{(\bigwedge^n f^\perp)^{-1}}^{-1} \norm{\mathbf{v} \wedge \mathbf{w}}}{\norm{\bigwedge^m f^\perp}\,\norm{\bigwedge^{n-m} f^\perp}}\\
&\geq \delta^{O(\epsilon)}\dang(V,W).
\end{align*}
Hence $f^\perp V \in \Vang(f^\perp W,\rho)$ implies $V \in \Vang(W,\delta^{-O(\epsilon)}\rho)$, which establishes \eqref{eq:GLonGr}.

For the second statement, observe that there is a finite set $\tilde A$ of cardinality $\abs{\tilde A} = \Ncov(\pi_V (A))$ such that
\[ A \subset \tilde A + V^{\perp} + \Ball(0,\delta).\]
Applying $f$ and then $\pi_{f^\perp V}$ on both sides, we obtain 
\[\pi_{f^\perp V}(fA) \subset \pi_{f^\perp V}(f\tilde A) + \Ball(0,\delta^{1-\epsilon}).\]
This proves that $\Ncov[\delta^{1-\epsilon}](\pi_{f^\perp V}(f A)) \leq \Ncov(\pi_V (A))$. We conclude by the scale change estimate~\eqref{eq:deltaPrim}.

For the next statement, it suffices to prove that $f^\perp V \in \EC(f A, O(\epsilon))$ whenever $V \in \EC(A,\epsilon)$. Indeed, let $V \in \EC(A,\epsilon)$. Then there exists $A' \subset A$ such that $\Ncov(A') \geq \delta^\epsilon\Ncov(A)$ and $\Ncov(\pi_V (A')) < \delta^{-\frac{m}{n}\alpha - \epsilon}$. On the one hand, by \eqref{eq:sizefA}, we have $\Ncov(fA') \geq \delta^{O(\epsilon)}\Ncov(fA)$. On the other hand, from \ref{it:GLpifA} it follows that $\Ncov(\pi_{f^\perp V}(f A')) \leq \delta^{-\frac{m}{n}\alpha - O(\epsilon)}$. Hence $f^\perp V \in \EC(fA,O(\epsilon))$.

For the last statement, it suffices to prove that for any $x \in f^\perp V$, there exists $y \in V$ such that
\begin{equation}\label{eq:fAcapfperp}
fA \cap \pi_{f^\perp V}^{-1}(x^{(\delta)}) \subset f\bigl( A \cap \pi_V^{-1}(y^{(\delta^{1-\epsilon})})\bigr).
\end{equation}
Indeed, if $a \in A$ satisfies $\pi_{f^\perp V}(f(a)) \in x^{(\delta)}$, then
\[f(a) \in x + f V^\perp +\Ball(0,\delta).\]
Applying $f^{-1}$ and then $\pi_V$ on both sides, we obtain
\[\pi_V(a) \in \pi_V(f^{-1}(x)) + \Ball(0,\delta^{1-\epsilon}).\]
This proves \eqref{eq:fAcapfperp} with $y = \pi_V(f^{-1}(x))$.
\end{proof}

\subsection{Non-concentration property for projections}
Let $A$ a subset of $\R^n$ as in Theorem~\ref{thm:main}. We want to understand whether a projection of $A$ still satisfies some similar regularity property as $A$ does. More precisely we want to find $V \in \Gr(\R^n,m)$ and a large subset $A'$ of $A$ such that 
\begin{equation*}
\forall \rho \geq \delta,\; \forall x \in V,\quad \Ncov(\pi_V(A')\cap x^{(\rho)}) \leq \rho^{\kappa_1} \delta^{-\frac{m}{n}\alpha - \epsilon'},
\end{equation*}
for some $\kappa_1 > 0$ proportional to $\kappa$ and some $\epsilon' > 0$ proportional to $\epsilon$. 

In the special case where $m$ divides $n$, we have the following result. We will only need this non-concentration result in this special case, although it might be true in a more general context.
\begin{lemm}\label{lm:NCproj}
Let $n = qm$ with $q \geq 2$. For any parameters $0 < \alpha < n$, $\kappa > 0$ and $\epsilon > 0$, the following is true for $\delta > 0$ sufficiently small. If $A$ is a subset of $\R^n$ contained in the unit ball and $\mu$ is a probability measure on $\Gr(\R^n,m)$ satisfying the assumptions \eqref{eq:sizeA}--\eqref{eq:nonconmu} for the parameters $\alpha$, $\kappa$ and $\epsilon$, then
\[\mu(\EC(A) \setminus \ECreg(A)) \leq \delta^{2\epsilon},\]
where $\ECreg(A)$ denotes the set of all $V \in \EC(A)$ such that $\exists A' \subset A$ with $\Ncov(A') \geq \delta^{2\epsilon}\Ncov(A)$ and $\Ncov(\pi_V(A')) \leq \delta^{-\frac{m}{n}\alpha -\epsilon}$ and
\begin{equation}\label{eq:NCprojA}
\forall \rho \geq \delta,\; \forall x \in V,\quad \Ncov(\pi_V(A')\cap x^{(\rho)}) \leq \rho^{\frac{\kappa}{2q^2}} \delta^{-\frac{m}{n}\alpha - 10\epsilon}.
\end{equation}
\end{lemm}

The idea of the proof is the following. When $V \in \EC(A)$, there is a large subset $A'$ with small projection to $V$. We then remove small fibers of the projection $\pi_V : A' \to V$ to get $A''$. Any large subset in of $\pi_V(A'')$ will have large preimage by $\pi_V$. Thus if $V \notin \ECreg(A)$ then there will be a cylinder with axis $V^\perp$ and radius $\rho$ in which $A$ is very dense. If there are a lot of such $V$ we can then intersect these cylinders to get a ball of radius $\rho^{\Inv{q}}$ which will contradict the non-concentration property \eqref{eq:nonconA} of $A$.

\begin{proof}
For conciseness, write $\kappa_1 = \frac{\kappa}{2q^2}$. We claim that if $V \in \EC(A) \setminus \ECreg(A)$ then there exists $x \in V$ and $\rho \geq \delta$ such that
\begin{equation}\label{eq:bigCyl}
\Ncov\bigl(A \cap \pi_{V}^{-1}(x^{(\rho)})\bigr) \geq \rho^{\kappa_1}\delta^{- 6\epsilon} \Ncov(A).
\end{equation}

Indeed, let $V \in \EC(A) \setminus \ECreg(A)$. Then from the definition~\eqref{eq:ECAeps} there exists $A' \subset A$ with $\Ncov(A') \geq \delta^{\epsilon}\Ncov(A)$ and 
$\Ncov(\pi_V(A')) \leq \delta^{-\frac{\alpha}{q} - \epsilon}$. Now we remove small fibers of the map $\pi_V$ restricted to $A'$. Consider the set
\[B = \bigl\{y \in V \mid \Ncov\bigl(A' \cap \pi_V^{-1}(y^{(\delta)})\bigr) \geq \delta^{\frac{\alpha}{q} + 3\epsilon}\Ncov(A)\bigr\}\]
and $A'' = A' \cap \pi_V^{-1}(B^{(\delta)})$. We have, for all $y \in V$,
\[\Ncov\bigl((A'\setminus\!A'') \cap \pi_V^{-1}(y^{(\delta)})\bigr) \leq \delta^{\frac{\alpha}{q} + 3\epsilon}\Ncov(A).\]
for otherwise $y$ would belong to $B$ and the intersection $(A'\setminus\! A'') \cap \pi_V^{-1}(y^{(\delta)})$ would be empty. Consequently,
\[\Ncov(A'\setminus\! A'') \leq \Ncov(\pi_V(A')) \max_{y \in V} \Ncov\bigl((A'\setminus\! A'') \cap \pi_V^{-1}(y^{(\delta)})\bigr) \leq\delta^{2\epsilon}\Ncov(A).\]
It follows that $\Ncov(A'') \geq \delta^{2\epsilon}\Ncov(A)$. But $V \notin \ECreg(A)$, the non-concentration property~\eqref{eq:NCprojA} fails for $\pi_V(A'')$ : there exists $x \in V$ and $\rho \geq \delta$ such that
\begin{equation}\label{eq:Vnotreg}
\Ncov\bigl(\pi_V(A'')\cap x^{(\rho)}\bigr) \geq \rho^{\kappa_1} \delta^{-\frac{\alpha}{q}-10\epsilon}.
\end{equation}
Let $\tilde B$ be a maximal $6\delta$-separated subset of $\pi_V(A'')\cap x^{(\rho)}$. From \eqref{eq:coverLeb} and \eqref{eq:Vnotreg}, we have $\abs{\tilde B} \gg \rho^{\kappa_1} \delta^{-\frac{\alpha}{q}-10\epsilon}$. Moreover for all $y \in \tilde B$, by the definition of $A''$, $y \in B^{(\delta)}$, hence $\Ncov\bigl(A' \cap \pi_{V}^{-1}(y^{(2\delta)})\bigr) \geq \delta^{\frac{\alpha}{q} + 3\epsilon}\Ncov(A)$. Since $\tilde B$ is $6\delta$-separated, all the balls $y^{(2\delta)}$ with center $y \in \tilde B$ are $2\delta$-away from each other. Consequently,
\begin{align*}
\Ncov\bigl(A' \cap \pi_{V}^{-1}(x^{(\rho + 2\delta)})\bigl) &\geq \sum_{y \in \tilde B} \Ncov\bigl(A'\cap \pi_{V}^{-1}(y^{(2\delta)})\bigr)\\
&\geq \abs{\tilde B} \delta^{\frac{\alpha}{q} + 3\epsilon}\Ncov(A) \geq (\rho+2\delta)^{\kappa_1}\delta^{- 6\epsilon}\Ncov(A).
\end{align*}
This finishes the proof of the claim.

To obtain a contradiction, suppose that $\mu(\EC(A) \setminus \ECreg(A)) \geq \delta^{2\epsilon}$. Note that the radius $\rho$ in the claim depends on $V$. Nevertheless, from \eqref{eq:bigCyl} we know that it ranges from $\delta$ to $\delta^{\frac{6\epsilon}{\kappa_1}}$. For the argument below, we want \eqref{eq:bigCyl} to hold for a lot of $V \in \EC(A) \setminus \ECreg(A)$ with some radius $\rho \geq \delta$ independent of $V$. Indeed, by a simple pigeonhole argument\footnote{Arrange different $\rho$ into intervals of the form $[\delta^{2^{-k}},\delta^{2^{-k-1}}]$, where $0 \leq k \ll -\log(\epsilon)$.}, we can find a subset $\DC \subset \EC(A) \setminus \ECreg(A)$ and a radius $\rho \geq \delta$ such that $\mu(\DC) \geq \delta^{3\epsilon}$ and for all $V \in \DC$, there exists $x \in V$ such that
\[\Ncov(A \cap \pi_{V}^{-1}(x^{(\rho)})) \geq \rho^{2\kappa_1}\delta^{-5\epsilon} \Ncov(A)\]
and hence, by Lemma~\ref{lm:coverLeb},
\[\lambda(A^{(\delta)} \cap \pi_{V}^{-1}(x^{(\rho)})) \geq \rho^{2\kappa_1}\delta^{-4\epsilon} \lambda(A^{(\delta)}).\]

Let $V_1,\dotsc,V_q$ be random elements of $\Gr(\R^n,m)$ independently distributed according to $\mu$. On the one hand, from Lemma~\ref{lm:bigCap} applied to the restriction of $\mu$ to $\DC$, it follows that with probability at least $\Inv{2}(\rho^{2\kappa_1}\delta^{-4\epsilon})^q \mu(\DC)^q \geq \Inv{2}\rho^{2q\kappa_1}\delta^{-q\epsilon}$, there exists $x_1 \in V_1, \dotsc , x_q \in V_q$ such that
\begin{equation}\label{eq:interCyl}
\lambda\bigl(A^{(\delta)} \cap \pi_{V_1}^{-1}(x_1^{(\rho)}) \cap \dotsb \cap \pi_{V_q}^{-1}(x_q^{(\rho)})\bigr) \geq \Inv{2}\rho^{2q\kappa_1}\delta^{-4q\epsilon} \lambda(A^{(\delta)}).
\end{equation}

On the other hand, from \eqref{eq:dangExpand} and \eqref{eq:nonconmu}, it follows that with probability at least $1 - (q-1)\delta^{-\epsilon}\rho^\frac{\kappa}{q}$, we have
\begin{equation}\label{eq:dangq1q}
\dang(V_1,\dotsc,V_q) \geq \rho^{\frac{q-1}{q}}.
\end{equation}

Now with our choice of $\kappa_1$, we have $1 - (q-1)\delta^{-\epsilon}\rho^\frac{\kappa}{q} + \Inv{2}\rho^{2q\kappa_1}\delta^{-q\epsilon} > 1$. This means that for some $(V_1,\dotsc,V_q)$, both \eqref{eq:interCyl} and \eqref{eq:dangq1q} hold. By Lemma~\ref{lm:Cylinders}, there exists $x \in \R^n$ such that 
\[\pi_{V_1}^{-1}(x_1^{(\rho)}) \cap \dotsb \cap \pi_{V_q}^{-1}(x_q^{(\rho)}) \subset x^{(\rho')}\]
with $\rho' = q \rho \dang(V_1,\dotsc,V_q)^{-1} \leq q\rho^{\Inv{q}}$. Then the non-concentration property~\eqref{eq:nonconA} of $A$ implies that
\[\lambda\bigl(A^{(\delta)} \cap x^{(\rho')}\bigr) \ll \delta^{-\epsilon}\rho^{\frac{\kappa}{q}}\lambda(A^{(\delta)}).\]
Combining this with \eqref{eq:interCyl} yields
\[\rho^{2q\kappa_1}\delta^{-4q\epsilon} \ll \delta^{-\epsilon}\rho^{\frac{\kappa}{q}},\]
which is impossible with our choice of $\kappa_1$.
\end{proof}

\subsection{Non-concentration property for slices}
We shall also consider slices of $A$, i.e. intersection of $A$ with a $\delta$-neighborhood of a affine subspace. When $n = qm$, we have similar non-concentration results for $(n- m)$\dash{}dimensional slices of $A$.
\begin{lemm}\label{lm:NCslice}
Let $n = qm$ with $q \geq 2$ a positive integer. Let $0< \alpha < n$, $\kappa > 0$ and $\epsilon > 0$ be parameters. If the statement in Theorem~\ref{thm:main} fails for the set $A$, then there is a $(n-m)$\dash{}dimensional affine subspace $y + W$ and a subset $B \subset A^{(\delta)}\cap (y + W)$ such that
\[\Ncov(B) \geq \delta^{-\beta+O(\epsilon)} \text{ and}\]
\begin{equation}\label{eq:Breg}
\forall \rho \geq \delta,\; \forall x \in W \quad \Ncov(B\cap x^{(\rho)}) \leq \rho^\frac{\kappa}{2q^2} \delta^{-\beta - O(\epsilon)},
\end{equation}
where $\beta = \frac{q-1}{q}\alpha$.
\end{lemm}

Here is an outline of the proof. The negation of Theorem~\ref{thm:main} to $A$ implies that there is a large subset $A_q \subset A$ occupying a large portion of the Cartesian product $\prod_{j=1}^q \pi_{V_j}(A_q)$ of its projections to $q$ subspaces in nearly orthogonal position. Then, because of Lemma~\ref{lm:NCproj}, the first factor $\pi_{V_1}(A_q)$ can be chosen to have the non-concentration property. This in turn will imply the non-concentration property of the projection of $A_q$ to $V_1 + \dotsb + V_{q-1}$. Then it would suffice to find a slice whose projection to $V_1 + \dotsb + V_{q-1}$ is nearly as large as that of $A_q$, which can be easily done given the negation of Theorem 6. 

\begin{proof} 
Suppose the statement in Theorem~\ref{thm:main} fails for the set $A\subset \R^n$. This means $\mu(\EC(A')) > \delta^\epsilon$ for any subset $A' \subset A$. In particular, $\ECreg(A)$ is non-empty by Lemma~\ref{lm:NCproj}. Let $V_1 \in \ECreg(A)$. There exists $A_1 \subset A$ with $\Ncov(A_1) \geq \delta^{-\alpha + 3\epsilon}$ and
\begin{equation}\label{eq:A1reg}
\forall \rho \geq \delta,\; \forall x \in V_1, \quad \Ncov(\pi_{V_1}(A_1)\cap x^{(\rho)}) \leq \rho^\frac{\kappa}{2q^2} \delta^{-\frac{1}{q}\alpha - 10\epsilon}.
\end{equation} 

Let $\epsilon_1 = \frac{3\epsilon}{\kappa}$. We construct by a simple induction a sequence of subspaces $V_2,\dotsc,V_q$ and a nested sequence of subsets $A_1 \supset \dotsb \supset A_q$ satisfying for any $j = 2, \dotsc, q$,
\begin{gather}
\label{eq:VjV1+}
\dang(V_j,V_1 + \dotsb + V_{j-1}) \geq \delta^{\epsilon_1},\\
\Ncov(A_j) \geq \delta^\epsilon \Ncov(A_{j-1}),\nonumber\\
\label{eq:VjAj<}
\Ncov(\pi_{V_j}(A_j)) \leq \delta^{-\frac{1}{q}\alpha - \epsilon}.
\end{gather}
This is possible since at each step, we have by \eqref{eq:nonconmu},
\[\mu\bigl(\EC(A_{j-1}) \setminus \Vang(V_1 + \dotsb + V_{j-1},\delta^{\epsilon_1})\bigr) \geq \delta^\epsilon - \delta^{2\epsilon} > 0.\]

From the fact that 
\[\Ncov(A_q) \leq \Ncov(\pi_{V_q}(A_q)) \max_{y \in V_q} \Ncov(A_q \cap \pi_{V_q}^{-1}(y^{(\delta)})),\]
we get some $y_\star \in V_q$ such that
\begin{equation}\label{eq:AqVqy}
\Ncov(A_q \cap \pi_{V_q}^{-1}(y_\star^{(\delta)})) \geq \delta^{-\frac{q-1}{q}\alpha + O(\epsilon)}.
\end{equation}
After a translation, we can suppose $y_\star = 0$. We write $V = V_1 + \dotsb + V_{q-1}$ and $W = V_q^\perp$ and set $B_0 = A_q \cap W^{(\delta)}$ and $B = \pi_W(B_0)$. We have $\Ncov(B) \geq \delta^{-\beta+O(\epsilon)}$ from \eqref{eq:AqVqy} and the fact that $B_0 \subset B^{(\delta)}$.

It remains to show the non-concentration property~\eqref{eq:Breg} for $B$. Let $\rho \geq \delta$ and $x \in W$. From \eqref{eq:VjV1+}, $\dang(V,W^\perp) = \dang(V,V_q) \geq \delta^{O(\epsilon)}$. Hence, by \eqref{eq:fromW2V} in Lemma~\ref{lm:fromVp2V},
\begin{equation*}
\Ncov(B \cap x^{(\rho)}) \leq \delta^{-O(\epsilon)}\Ncov(\pi_V(B) \cap x_0^{(\rho)})
\end{equation*}
where $x_0 = \pi_V(x)$. Moreover $B \subset A_q^{(\delta)}$, hence, by \eqref{eq:deltacap},
\begin{equation*}
\Ncov(\pi_V(B) \cap x_0^{(\rho)}) \leq \delta^{-\epsilon}\Ncov(\pi_V(A_q) \cap x_0^{(2\rho)}).
\end{equation*}

Then Lemma~\ref{lm:tiltCart} applied to the set $\pi_V(A_q) \cap x_0^{(2\rho)}$ in $V=\bigoplus_{j=1}^{q-1} V_j$ together with \eqref{eq:VjV1+} yield
\begin{equation*}
\Ncov(\pi_V(A_q) \cap x_0^{(2\rho)}) \leq \delta^{-O(\epsilon)} \Ncov(\pi_{V_1}(A_q) \cap x_1^{(2\rho)})\prod_{j = 2}^{q-1} \Ncov(\pi_{V_j}(A_q))
\end{equation*}
where $x_1 = \pi_{V_1}(x_0)$. The required non-concentration property~\eqref{eq:Breg} then follows from \eqref{eq:A1reg} and \eqref{eq:VjAj<}.
\end{proof}

\subsection{Without the non-concentration property} As illustrated by the example in the introduction, the non-concentration condition~\eqref{eq:nonconA} on $A$ is crucial to have a gain $\epsilon > 0$ in the conclusion. Without this condition, we can still expect $\Ncov(\pi_V(A))$ to be close to $\Ncov(A)^\frac{m}{n}$ for generic $V \in \Gr(\R^n,m)$. This is the subject of the next proposition.

\begin{prop}\label{pr:Nincr}
Given $0 < m \leq n$, $0< \alpha <n$ and $\kappa > 0$, there exists $C > 1$ such that for all $0 < \epsilon < \frac{\kappa}{C}$, the following is true for all $\delta > 0$ sufficiently small. Let $A \subset \R^n$ be a subset contained in the unit ball and $\mu$ a probability measure on $\Gr(\R^n,m)$. Assume that
\begin{equation}\label{eq:sizeApp}
\Ncov(A) \geq \delta^{-\alpha - C\epsilon}.
\end{equation}
Further assume the non-concentration property \eqref{eq:nonconmu} for $\mu$ if $m < n$. Then
\[\mu(\EC(A)) \leq \delta^{\epsilon}.\]
\end{prop}

When $m$ divides $n$, this follows almost immediately from Lemma~\ref{lm:tiltCart}. Then the task is to reduce to this special case. Since it shares the same set of ideas as the proof of Theorem~\ref{thm:main}, the proof below will only be outlined and more details can be found in the next section.

\begin{proof}
For $0 < m \leq n$, denote by $\PC(n,m)$ the statement we want to show. Note that for all $n \geq 1$, $\PC(n,n)$ is trivially true. We will proceed by an induction similar to that in the proof of Theorem~\ref{thm:main}. It suffices to show the following two types of inductive steps. Let $0 < m \leq n$ and $q,r > 0$ be integers.
\begin{enumerate}
\item \label{it:Nincr1} If $mq \leq n$, then $\PC(n,qm)$ implies $\PC(n,m)$.
\item \label{it:Nincr2} If $n = q(n-m) + r$ with $0 < r \leq n - m$, then $\PC(n,r)$ and $\PC(n-r,m)$ imply $\PC(n,m)$.
\end{enumerate}

Using the same argument in Proposition~\ref{pr:unionAps}, we see that in order to show $\PC(n,m)$, it suffices to show $\mu(\EC(A')) \leq \delta^\epsilon$ for some subset $A' \subset A$. In other words, if the conclusion of $\PC(n,m)$ fails for the set $A$ then for any subset $A' \subset A$, $\mu(\EC(A')) \geq \delta^\epsilon$.

\emph{Proof of \ref{it:Nincr1}}. Let $V_1,\dotsc,V_q$ be random elements of $\Gr(\R^n,m)$ independently distributed according to $\mu$.
Write $V=V_1 + \dotsb + V_q$. When $qm < n$, we know by Lemma~\ref{lm:sumsOK} that
\[\Prob{\dim(V) = qm} \geq 1 - (q-1)\delta^{\kappa - \epsilon},\]
and the distribution of $V$ conditional to the event $\dim(V) = qm$ has the corresponding non-concentration property. By $\PC(n,qm)$, we know that for any $C' >0$, if the constant $C$ in \eqref{eq:sizeApp} is large enough (depending on $C'$) then the probability that there exists $A'\subset A$ satisfying
\[\Ncov(A')\geq \delta^{C'\epsilon}\Ncov(A) \text{ and } \Ncov(\pi_V(A')) \leq \delta^{-\frac{qm}{n}\alpha-C'\epsilon}\]
is at most $\delta^{C'\epsilon} + (q-1)\delta^{\kappa - \epsilon}$.

Suppose that $\PC(n,m)$ fails for $A$. Then by a simple induction we show that with probability at least $\delta^{O(\epsilon)}$, we have
\[\dang(V_1,\dotsc,V_q) \geq \delta^{O(\epsilon)}\]
and there exists $A_q \subset A$ such that $\Ncov(A_q)\geq \delta^{O(\epsilon)}\Ncov(A)$ and 
\[\forall j = 1,\dotsc,q,\quad \Ncov(\pi_{V_j}(A_q)) \leq \delta^{-\frac{m}{n}\alpha-\epsilon}\]
and hence, by \eqref{eq:tiltCart} applied to $\pi_V(A)$ in $V=\bigoplus_{j=1}^{q} V_j$,
\[\Ncov(\pi_V(A)) \leq \delta^{-\frac{qm}{n}\alpha-O(\epsilon)}.\]
We obtain a contradiction if $C'$ were chosen to be larger than any of the implicit constants in the Landau notations appearing above.

\emph{Proof of \ref{it:Nincr2}, Case 1}. Assume firstly that $A$ contains large slice of dimension $n-r$. More precisely, assume that there exists $W \in \Gr(\R^n,n-r)$ and $x \in \R^n$ such that
\[ \Ncov\bigl(A \cap (x + W^{(\delta)})\bigr) \geq \delta^{-\frac{n-r}{n}\alpha - C'\epsilon}\] 
where $C'$ is the constant given by $\PC(n-r,m)$ applied to $0< m \leq n-r$, $\frac{n-r}{n}\alpha$ and $\kappa$. Without loss of generality, we can assume that $x = 0$ and that $B = \pi_W(A \cap W^{(\delta)})$ is contained in $A$. Lemma~\ref{lm:piWVok} tells us that we can apply $\PC(n-r,m)$ to $B \subset W$ with the image measure of $\mu$ by $\pi_W$. Then we can conclude using Lemma~\ref{lm:piVpB}.

\emph{Proof of \ref{it:Nincr2}, Case 2}. Otherwise $A$ does not contain any large slice of dimension $n-r$ :
\begin{equation}\label{eq:NOn-rSlice}
\forall x \in \R^n,\; \forall W \in \Gr(\R^n,n-r),\quad \Ncov\bigl(A \cap (x + W^{(\delta)})\bigr) \leq \delta^{-\frac{n-r}{n}\alpha - O(\epsilon)}.
\end{equation}
Let $V_1,\dotsc,V_q$ be random elements of $\Gr(\R^n,m)$ independently distributed according to $\mu$.
Write $V=V_1 \cap \dotsb \cap V_q$. By \eqref{eq:perpperp} and Lemma~\ref{lm:sumsOK} applied to $V_1^\perp + \dotsb + V_q^{\perp} = V^\perp$,
\[\Prob{\dim(V) = r} \geq 1 - (q-1)\delta^{\kappa - \epsilon}\]
and that the distribution of $V$ conditional to the event $\dim(V) = r$ has a non-concentration property. By $\PC(n,r)$, we know that for any $C' >0$, if the constant $C$ in \eqref{eq:sizeApp} is large enough (depending on $C'$) then the probability that there exists $A'\subset A$ satisfying
\[\Ncov(A')\geq \delta^{C'\epsilon}\Ncov(A) \text{ and } \Ncov(\pi_V(A')) \leq \delta^{-\frac{r}{n}\alpha-C'\epsilon}\]
is at most $\delta^{C'\epsilon} + (q-1)\delta^{\kappa - \epsilon}$.

Suppose that $\PC(n,m)$ fails for $A$. Again by an induction we show that with probability at least $\delta^{O(\epsilon)}$, we have
\[\dang(V_1^\perp,\dotsc,V_q^\perp) \geq \delta^{O(\epsilon)}\]
and there exists $A_q \subset A$ such that $\Ncov(A_q)\geq \delta^{O(\epsilon)}\Ncov(A)$ and 
\[\forall j = 1,\dotsc,q,\quad \Ncov(\pi_{V_j}(A_q)) \leq \delta^{-\frac{m}{n}\alpha-\epsilon}\]
Together with \eqref{eq:NOn-rSlice}, this implies by Proposition~\ref{pr:SliceEndelta} that there exists $A' \subset A_q$ such that
\[\Ncov(A') \geq \delta^{O(\epsilon)} \Ncov(A) \text{ and } \Ncov(\pi_V(A')) \leq \delta^{-\frac{r}{n}\alpha - O(\epsilon)}.\]
Again we obtain a contradiction if $C'$ is large compared to any of the implied constants in the previous Landau notations.
\end{proof}

%% file: proof.tex
\section{Proof of the main result}
In this section, we prove Theorem~\ref{thm:main} and thus Theorem~\ref{thm:proj}. This is done by proving first the base case where $n = 2m$ (Propsoition~\ref{pr:n2mcase}) and then the induction steps (Propositions~\ref{pr:nqmcase}-\ref{pr:nqn-mrcase}). Note that on account of Proposition~\ref{pr:unionAps}, for a given pair $(n,m)$, if Theorem~\ref{thm:main} is true for these dimensions then so is Theorem~\ref{thm:proj}. Therefore, when we use Theorem~\ref{thm:main} as induction hypothesis, the conclusion is $\mu(\EC(A)) \leq \delta^\epsilon$ while when we prove by contradiction by saying that $A$ is a counterexample for Theorem~\ref{thm:main}, we are assuming $\mu(\EC(A')) > \delta^{\epsilon}$ for all subsets $A'$ of $A$.

Like in the previous section, all implied constants in Landau and Vinogradov notations in this section may depend on $n$ and $\kappa$. Again, every statement in this section is true only for $\delta > 0$ smaller than a constant depending on $n$, $m$, $\alpha$, $\kappa$ and $\epsilon$.

\subsection{Half dimensional projections}\label{ss:n2mcase}
For the special case $n=2m$, we follow mainly the proof in \cite{Bourgain2010} (which deals with the case $m=1$) while using a technique in the proof of Proposition 2 in Bourgain-Glibichuk~\cite{BourgainGlibichuk}. The main idea, as explained in the introduction, is to reduce to the case where $A$ is a Cartesian product $X_\star \times X_\star$ with the help of Balog-Szemerédi-Gowers theorem and then apply a sum-product estimate.

\begin{proof}[Proof of Proposition~\ref{pr:n2mcase}]
Suppose Theorem~\ref{thm:main} fails for the subset $A \subset \R^{n}$ and the probability measure $\mu$ on $\Gr(\R^n,m)$ with $n=2m$. We will get a contradiction when $\epsilon$ is small enough. By Lemma~\ref{lm:NCproj}, there is a subspace $V_1$ and a subset $A_1 \subset A$ with the following properties:
\[\Ncov(A_1) \geq \delta^{-\alpha + 3\epsilon} , \;
\Ncov(\pi_{V_1}(A_1)) \leq \delta^{-\frac{\alpha}{2} - \epsilon} \;\text{and}\]
\begin{equation}\label{eq:nonconXstar}
\forall \rho \geq \delta,\; \forall x \in V_1,\; \Ncov(\pi_{V_1}(A_1) \cap x^{(\rho)}) \leq \rho^{\frac{\kappa}{8}} \delta^{-\frac{\alpha}{2} - O(\epsilon)}.
\end{equation}

Let $\epsilon_1 = \frac{3\epsilon}{\kappa}$. Then $\mu\bigl(\EC(A_1) \setminus \Vang(V_1,\delta^{\epsilon_1})\bigr) \geq \delta^\epsilon - \delta^{2\epsilon} > 0$ by the non-concentration property \eqref{eq:nonconmu} of $\mu$. Let $V_2 \in \EC(A_1) \setminus \Vang(V_1,\delta^{\epsilon_1})$ with $A_2$ such that
\begin{equation}\label{eq:piV12A2}
\Ncov(A_2) \geq \delta^{-\alpha + 4\epsilon}
\; \text{and} \;
\Ncov(\pi_{V_j}(A_2)) \leq \delta^{-\frac{\alpha}{2} - \epsilon} ,\qquad j = 1,2.
\end{equation}

Observe that since $\dang(V_1,V_2)\geq \delta^{O(\epsilon)}$ there exists a $\delta^{-O(\epsilon)}$-bi-Lipschitz map $f \in \GL(\R^n)$ satifying (recalling the notation introduced in Subsection~\ref{ss:actGL}) $f^\perp V_1 = V_1$ and $f^\perp V_2 = V_1^\perp$. Applying Lemma~\ref{lm:GLaction} to $f$, we see that we can assume without loss of generality that $V_2 = V_1^\perp$. 

Put $X = \pi_{V_1}(A_2)$ and $Y = \pi_{V_2}(A_2)$. We have, $\Ncov(A_2) \ll \Ncov(X)\Ncov(Y)$ and this together with the inequalities~\eqref{eq:piV12A2} implies
\[\Ncov(X),\, \Ncov(Y) \geq \delta^{-\frac{\alpha}{2} + O(\epsilon)}.\]

Write $\DC = \EC(A_2) \setminus (\Vang(V_1,\delta^{\epsilon_1}) \cup \Vang(V_2,\delta^{\epsilon_1}))$. We have, by~\eqref{eq:nonconmu}, $\mu(\DC) \geq \delta^\epsilon -2\delta^{2\epsilon} \geq \delta^{2\epsilon}$. Let $V \in \DC$. By \eqref{eq:det=dang} and \eqref{eq:perpperp}, we have
\[\abs{\det(\pi_{V\mid V_1})} = \dang(V_1,V^\perp) = \dang(V_2,V) \geq \delta^{O(\epsilon)}.\]
The same is true for $\pi_{V\mid V_2}$. Then it follows easily from the Cartan decomposition that
\begin{equation}\label{eq:normInvPi}
\norm{\pi_{V\mid V_1}^{-1}} \leq \delta^{-O(\epsilon)} \text{ and } \norm{\pi_{V\mid V_2}^{-1}} \leq \delta^{-O(\epsilon)}.
\end{equation}

Since $V \in \EC(A_2)$, there is a subset $A_V \subset A_2$ such that $\Ncov(A_V) \geq \delta^{-\alpha + O(\epsilon)}$ and
$\Ncov(\pi_V(A_V)) \leq \delta^{-\frac{\alpha}{2} - \epsilon}$. It follows from~\eqref{eq:phiEnergy} that
\[\En_\delta(\pi_V,X + Y) \geq \En_\delta(\pi_V,A_V) \geq \delta^{-\frac{3\alpha}{2} + O(\epsilon)}.\]

By~\eqref{eq:normInvPi}, the map $\R^n = V_1 \oplus V_2 \to V \times V$ defined by $v_1 + v_2 \mapsto (\pi_V(v_1), \pi_V(v_2))$ is $\delta^{-O(\epsilon)}$-bi-Lipschitz. Hence, by \eqref{eq:Enphipsi}, we can bound from below the additive energy between $\pi_V X$ and $\pi_V Y$,
\[\En_\delta(+, \pi_V X \times \pi_V Y) \geq \delta^{-\frac{3\alpha}{2} + O(\epsilon)} \geq \delta^{O(\epsilon)} \Ncov(\pi_V X)^{\frac{3}{2}} \Ncov(\pi_V Y)^{\frac{3}{2}}.\]

That is why we can apply the Balog-Szemerédi-Gowers theorem (Theorem~\ref{thm:BSG+}) to get subsets $X_V \subset X$ and $Y_V \subset Y$ such that
\begin{equation}\label{eq:sizeXV}
\Ncov(X_V),\,\Ncov(Y_V) \geq \delta^{-\frac{\alpha}{2} + O(\epsilon)}
\end{equation}
and
\begin{equation}\label{eq:BVtimesCV}
\Ncov(\pi_V X_V + \pi_V Y_V) \leq \delta^{-\frac{\alpha}{2} - O(\epsilon)}.
\end{equation}

Applying $\pi_{V\mid V_1}^{-1}$ to the set in the last inequality and using \eqref{eq:sizefA}, we obtain
\begin{equation}\label{eq:BVplusfCV}
\Ncov(X_V + \phi_V Y_V) \leq \delta^{-\frac{\alpha}{2} - O(\epsilon)},
\end{equation}
where $\phi_V \colon V_2 \to V_1$ is $\phi_V = \pi_{V\mid V_1}^{-1} \circ \pi_{V\mid V_2}$. Note that from~\eqref{eq:normInvPi}, $\phi_V$ is $\delta^{-O(\epsilon)}$-bi-Lipschitz.

Let us apply Lemma~\ref{lm:bigCap} to the collection of subsets $X_V^{(\delta)} \times Y_V^{(\delta)} \subset X^{(\delta)} \times Y^{(\delta)}$ with the restriction of $\mu$ to $\DC$. We obtain $V_\star \in \DC$, $X_\star \coloneqq X_{V_\star}$ and $Y_\star \coloneqq Y_{V_\star}$ such that
\[\lambda(X_\star^{(\delta)} \cap X_V^{(\delta)})\lambda(Y_\star^{(\delta)} \cap Y_V^{(\delta)}) \geq \delta^{n - \alpha + O(\epsilon)}\]
whenever $V \in \DC'$, where $\DC'$ is a subset of $\DC$ with
\begin{equation}\label{eq:muDCp}
\mu(\DC') \geq \delta^{O(\epsilon)}\mu(\DC) \geq \delta^{O(\epsilon)}.
\end{equation}

By Ruzsa's triangular inequality (Lemma~\ref{lm:RuzsaTri}), \eqref{eq:BVplusfCV} implies, for all $V \in \DC'$
\[\Ncov(X_V - X_\star^{(\delta)} \cap X_V^{(\delta)}) \ll \Ncov(X_V - X_V) \leq \delta^{-\frac{\alpha}{2} - O(\epsilon)}.\]
For the same reason $\Ncov(X_\star - X_\star^{(\delta)} \cap X_V^{(\delta)}) \leq \delta^{-\frac{\alpha}{2} - O(\epsilon)}$. Then by Ruzsa's triangular inequality again, we have
\begin{equation}\label{eq:RuszaT1}
\Ncov(X_\star - X_V) \leq \delta^{-\frac{\alpha}{2} - O(\epsilon)}.
\end{equation}
Similarly, $\Ncov(Y_\star - Y_V) \leq \delta^{-\frac{\alpha}{2} - O(\epsilon)}$, which implies with \eqref{eq:sizefA},
\begin{equation}\label{eq:RuszaT2}
\Ncov(\phi_V Y_\star - \phi_V Y_V) \leq \delta^{-\frac{\alpha}{2} - O(\epsilon)}.
\end{equation}
Moreover, \eqref{eq:BVplusfCV} specified to $V = V_*$ with \eqref{eq:sizefA} applied to $\phi_V \phi_\star^{-1}$ gives
\begin{equation}\label{eq:RuszaT3}
\Ncov(\phi_V \phi_\star^{-1} X_\star + \phi_V Y_\star) \leq \delta^{-\frac{\alpha}{2} - O(\epsilon)},
\end{equation}
where $\phi_\star \coloneqq \phi_{V_\star}$.

Now successive use of Ruzsa's triangular inequality (recalling \eqref{eq:RuszaT1}, \eqref{eq:BVplusfCV}, \eqref{eq:RuszaT2} and \eqref{eq:RuszaT3}) yields that for all $V \in \DC'$,
\begin{equation}\label{eq:XstarfXstar}
\Ncov(X_\star - \phi_V \phi_\star^{-1} X_\star) \leq \delta^{-\frac{\alpha}{2} - O(\epsilon)}.
\end{equation}
Moreover, by the Plünnecke-Ruzsa inequality (Lemma~\ref{lm:RuzsaSum}),
\begin{equation}\label{eq:XstarXstar}
\Ncov(X_\star + X_\star) \leq \delta^{-\frac{\alpha}{2} - O(\epsilon)}.
\end{equation}

Consider the set of endomorphisms $\AC = \{-\phi_V \phi_\star^{-1} \in \End(V_1) \mid V \in \DC'\}$. We claim that the assumptions of Theorem~\ref{thm:ActionRn} are satisfied for $\AC$ and $X_\star$ with $\sigma = \frac{\alpha}{2}$, and $\kappa$ replaced by $\frac{\kappa}{8}$ and $\epsilon$ replaced by $O(\epsilon)$ . Therefore, when $\epsilon$ is small enough, \eqref{eq:XstarfXstar} and \eqref{eq:XstarXstar} contradict Theorem~\ref{thm:ActionRn}.

Our claim about the assumptions~\ref{it:AinBall}, \ref{it:XinBall} and \ref{it:sizeXleq} are clear from what precedes. The assumption \ref{it:sizeXatrho} follows from \eqref{eq:nonconXstar} and \eqref{eq:sizeXV} because for any $\rho\geq \delta$,
\[ \Ncov(X_\star) \leq \Ncov[\rho](X_\star) \max_{x \in V_1} \Ncov(X_\star \cap x^{(\rho)}).\]

In the case of $m = 1$, the assumption \ref{it:irreducible} is trivially true and the assumption \ref{it:sizeAatrho} follows immediately from \eqref{eq:nonconmu} the fact that $\dang$ is a distance on $\Gr(\R^2,1)$ and the fact that the map $\Gr(\R^2,1) \setminus \Vang(V_2,\delta^{\epsilon_1}) \to \R$, $V \mapsto \phi_V\phi_\star^{-1}$ is $\delta^{-O(\epsilon)}$-bi-Lipschitz.

Finally, to prove \ref{it:sizeAatrho} and \ref{it:irreducible} in the case where $m \geq 2$ we use Lemma~\ref{lm:phiVrich} below. For any $f \in \End(V_1)$, pick an arbitrary nonzero vector $v_1 \in V_1$ and apply Lemma~\ref{lm:phiVrich} to $v_2 = \phi_\star^{-1}(v_1)$ and $W= \R f(v_1)$. This gives the existence of a subspace $W' \in \Gr(\R^n,m)$ such that $\dang(V,W') \leq \delta^{-O(\epsilon)}d(\phi_V\phi_\star^{-1},f)$. Hence by \eqref{eq:nonconmu}, for any $\rho \geq \delta$,
\[\mu\bigl( \{V \in \DC'\mid -\phi_V \phi_\star^{-1}\in \Ball(f,\rho)\}\bigr) \leq \delta^{-O(\epsilon)}\rho^\kappa.\]
Observe that
\[\mu(\DC') \leq \Ncov[\rho](\AC) \max_{f\in \End(V_1)} \mu\bigl( \{V \in \DC'\mid -\phi_V \phi_\star^{-1}\in \Ball(f,\rho)\}\bigr).\]
Together with \eqref{eq:muDCp}, this gives the assumption~\ref{it:sizeAatrho}, namely,
\[\Ncov[\rho](\AC) \geq \delta^{O(\epsilon)} \rho^{-\kappa}.\]

Moreover, for any nonzero proper linear subspace $W \in V_1$, take $w \in W$ some vector with $\norm{w}=1$ and consider
\[\rho_0 = \sup_{V \in \DC'} d(-\phi_V \phi_\star^{-1}(w),W).\]
By Lemma~\ref{lm:phiVrich} applied to $v_2 = \phi_\star^{-1}(w)$ which has norm $\leq \delta^{-O(\epsilon)}$, we have $\DC' \subset \Vang(W',\delta^{-O(\epsilon)}\rho_0)$ for some $W' \in \Gr(\R^n,m)$. In view of \eqref{eq:muDCp} and \eqref{eq:nonconmu}, we have $\delta^{O(\epsilon)} \leq \delta^{-O(\epsilon)}\rho_0^\kappa$. Hence $\rho_0 \geq \delta^{O(\epsilon)}$, which establishes \ref{it:irreducible}.
\end{proof}

\begin{lemm}\label{lm:phiVrich}
We use the notations in the proof above. For any nonzero vector $v_2 \in V_2$ and any proper linear subspace $W \subset V_1$, there is $W' \in \Gr(\R^n,m)$ such that for all $V \in \Gr(\R^n,m)$, 
\begin{equation}\label{eq:vawayW}
\dang(V,W') \leq \norm{v_2}^{-1}d(\phi_V(v_2), W) .
\end{equation}
\end{lemm}

\begin{proof}
Without loss of generality, we can assume that $\dim(W) = m -1$. For any $V \in \Gr(\R^n,m)$, any $v_2 \in V_2$ and any $w \in W$, by~\eqref{eq:wedgePi2}, we have
\[\dang(V^\perp,\R(v_2 - w)) = \frac{\norm{\pi_V(v_2 - w)}}{\norm{v_2 - w}}.\]
Note that $\norm{v_2 - w} \geq \norm{v_2}$ since $v_2 \perp w$ and $\norm{\pi_V(v_2 -w)} \leq \norm{\phi_V(v_2) - w}$ since $\pi_V(\phi_V(v_2) - w) = \pi_V(v_2 -w)$. Hence 
\[\dang(V^\perp,\R(v_2 - w)) \leq \frac{\norm{\phi_V(v_2) - w}}{\norm{v_2}}.\]
As $w$ can be any vector in $W$, we obtain 
\[\dang(V^\perp,\R v_2 + W) \leq \norm{v_2}^{-1}d(\phi_V(v_2),W).\]

We conclude by setting $W' = (\R v_2 + W)^\perp \in \Gr(\R^n,m)$ and using \eqref{eq:perpperp}.
\end{proof}

\subsection{Projection of a slice}\label{ss:nqmcase}
If the set $A$ contains a relatively large slice of dimension $0 < n' < n$ (a subset $B = A^{(\delta)} \cap (y + W)$ with $\dim(W) = n'$ and $\Ncov(B) \asymp \delta^{-\frac{n'}{n}\alpha}$) and if it has a correct non-concentration property then we can apply the induction hypothesis to $B - y$ inside $W$. Instead of projecting to $V$ distributed according to $\mu$, we project to $V' = \pi_W(V)$.
The first lemma below shows that $V'$ is not concentrated and the next one shows the relationship between $V'$ being in $\EC(B)\cap \Gr(W,m)$ and $V$ being in $\EC(B)$. Using this idea we prove Proposition~\ref{pr:nqmcase}.
\begin{lemm}\label{lm:piWVok}
Let $0 < m < n' < n$ be integers and $\kappa, \epsilon > 0$ be parameters. Let $W \in \Gr(\R^n,n')$ and $V$ be a random element of $\Gr(\R^n,m)$ having the following non-concentration property,
\begin{equation}\label{eq:nonconmuProb}
\forall \rho \geq \delta,\; \forall U \in \Gr(\R^n,n - m),\quad \Prob{\dang(V,U) \leq \rho} \leq \delta^{-\epsilon}\rho^\kappa.
\end{equation}
Set $V' = \pi_W(V)$. Then with probability at least $1 - \delta^{\kappa-\epsilon}$, $\dim(V') = m$. Conditional to this event the distribution of $V'$ is a probability measure $\nu$ on $\Gr(W,m)$. It satisfies
\[\forall \rho \geq \delta,\; \forall U \in \Gr(W,n'-m),\quad \nu(\Vang(U,\rho)) \leq \delta^{-2\epsilon} \rho^\kappa.\]
\end{lemm}

\begin{proof}
We know that $\dim(V') = m$ if and only if $\dang(V,W^\perp) > 0$. The first part follows immediately from~\eqref{eq:nonconmuProb} specified to $\rho = \delta$.

Let us show the non-concentration property for $\nu$. Let $U$ be a $(n' - m)$\dash{}dimensional subspace of $W$. 
By Lemma~\ref{lm:V2WandU}, we have $\dang(V,U+W^\perp) \leq \dang(V',U)$. Hence for all $\rho \geq \delta$, by \eqref{eq:nonconmuProb}
\[\Prob{\dang(V',U) \leq \rho} \leq \Prob{\dang(V,U+W^\perp) \leq \rho} \leq \delta^{-\epsilon}\rho^\kappa\]
and hence $\nu(\Vang(U,\rho)) \leq \frac{\delta^{-\epsilon} \rho^\kappa}{1 - \delta^{\kappa - \epsilon}} \leq \delta^{-2\epsilon} \rho^\kappa$.
\end{proof}

\begin{lemm}\label{lm:piVpB}
Let $0 < m \leq n' < n$ be integers. Let $0< \alpha < n$ and $\epsilon > 0$ be parameters. Let $B \subset W$ be a bounded subset in a $n'$\dash{}dimensional linear subspace $W \subset \R^n$. Then 
\[\pi_W\bigl(\EC(B,\epsilon) \setminus \Vang(W^\perp,\delta^\epsilon)\bigr) \subset \EC(B,O(\epsilon)) \cap \Gr(W,m).\]
\end{lemm}
\begin{proof}
Let $V \in \EC(B,\epsilon) \setminus \Vang(W^\perp,\delta^\epsilon)$. Then there exists $B' \subset B$ such that 
\[\Ncov(B') \geq \delta^\epsilon\Ncov(B) \text{ and } \Ncov(\pi_V(B')) \leq \delta^{-\frac{m}{n}\alpha - \epsilon}.\]
Denote by $V'$ the projection $\pi_W(V)$. It follows from Lemma~\ref{lm:fromVp2V} that
\[\Ncov(\pi_{V'}(B')) \leq \dang(V,W^\perp)^{-O(1)}\Ncov(\pi_V(B')) \leq \delta^{-\frac{m}{n}\alpha - O(\epsilon)}.\]
That is why $V' \in \EC(B,O(\epsilon)) \cap \Gr(W,m)$.
\end{proof}

\begin{proof}[Proof of Proposition~\ref{pr:nqmcase}]
Let $n = qm$ and suppose that Theorem~\ref{thm:main} holds for $n' = (q-1)m$ and $m$. Let $A$ and $\mu$ be as in Theorem~\ref{thm:main} but for which the conclusion fails. By Lemma~\ref{lm:NCslice}, there is an $n'$\dash{}dimensional affine subspace $y + W$ and a subset $B \subset A^{(\delta)} \cap (y + W)$ such that
\[\Ncov(B) \geq \delta^{-\beta+O(\epsilon)} \quad\text{and}\]
\[\forall \rho \geq \delta,\; \forall x \in W,\quad \Ncov(B\cap x^{(\rho)}) \leq \rho^\frac{\kappa}{2q^2} \delta^{-\beta - O(\epsilon)}\]
where $\beta = \frac{q-1}{q}\alpha$. Without loss of generality, we can assume $y = 0$ and $B \subset A$.

Let $V$ be a random element of $\Gr(\R^n,m)$ distributed according to $\mu$. Define $\nu$ be as in Lemma~\ref{lm:piWVok}. By the lemma, we can apply the induction hypothesis (Theorem~\ref{thm:main} combined with Proposition~\ref{pr:unionAps}) to $B\subset W$ with the probability measure $\nu$ on $\Gr(W,m)$. We obtain a constant $\epsilon' > 0$ depending only on $n',\beta$ and $\kappa$ such that when $\epsilon \leq \epsilon'$, 
\[\nu\bigl(\EC(B,\epsilon') \cap \Gr(W,m)\bigr) \leq \delta^{\epsilon'}.\]

Set $\epsilon_1 = \frac{3\epsilon}{\kappa}$. By Lemma~\ref{lm:piVpB}, we have
\[\mu\bigl(\EC(B,\epsilon_1) \setminus \Vang(W^\perp,\delta^{\epsilon_1})\bigr) \leq \nu\bigl(\EC(B,O(\epsilon)) \cap \Gr(W,m)\bigr).\]
When $\epsilon \leq \frac{\epsilon'}{O(1)}$, the last two inequalities together with~\eqref{eq:nonconmu} yield
\[\mu(\EC(B,\epsilon)) \leq \mu\bigl(\EC(B,\epsilon_1) \setminus \Vang(W^\perp,\delta^{\epsilon_1})\bigr)  + \mu(\Vang(W^\perp,\delta^{\epsilon_1})) \leq \delta^{\epsilon'} + \delta^{2\epsilon} \leq \delta^\epsilon,\]
which finishes the proof of Proposition~\ref{pr:nqmcase}.
\end{proof}

\subsection{Projection to a sum of subspaces}\label{ss:nqmrcase}
In the situation where $m < \frac{n}{2}$, we consider the sum $V = V_1 + \dotsb + V_q$ where $q$ is a positive integer such that $qm < n$ and $V_1,\dotsc,V_q$ are $m$\dash{}dimensional subspaces. Using the inequality~\eqref{eq:tiltCart}, the size of the projection to $V$ can be bounded in terms of the sizes of the projections to each $V_j$. In the next lemma, we prove that if $V_j$ are independently randomly distributed according to a measure with an appropriate non-concentration property then the distribution of their sum $V$ has a non-concentration property as well. This allows us to apply the induction hypothesis with the dimensions $n$ and $m' = qm$. This idea leads to the proof of Proposition~\ref{pr:nqmrcase}.

\begin{lemm}\label{lm:sumsOK}
Let $n,m,q,r$ be positive integers such that $qm + r = n$. Let $0< \epsilon < \frac{1}{2}\kappa$ be parameters. Let $V_1,\dotsc V_q$ be independent random elements of $\Gr(\R^n,m)$ satisfying $\forall j = 1, \dotsc, q$,
\[\forall \rho \geq \delta,\; \forall W \in \Gr(\R^n, n-m)\quad \Prob{\dang(V_j,W) \leq \rho} \leq \delta^{-\epsilon} \rho^\kappa.\]
Then with probability at least $1 - (q-1)\delta^{\kappa - \epsilon}$, we have
\begin{equation}\label{eq:sumsOK}
\dim(V_1 + \dotsb + V_q) = qm.
\end{equation}
Then the probability measure $\mu'$ on $\Gr(\R^n,qm)$ defined as the distribution of $V_1 + \dotsb + V_q$ conditional to the event \eqref{eq:sumsOK} satisfies the non-concentration property
\[\forall \rho \geq \delta,\; \forall W \in \Gr(\R^n,r), \quad \mu'(\Vang(W,\rho)) \leq \delta^{-O(\epsilon)}\rho^\frac{\kappa}{q}.\] 
\end{lemm}

\begin{proof}
Let $V_1,\dotsc,V_q$ be as in the statement. By their independence, for every $j = 2,\dotsc,q$,
\[\Prob{\dang(V_j,V_1 + \dotsb + V_{j-1}) \leq \delta} \leq \delta^{\kappa - \epsilon}.\]
Hence, on account of \eqref{eq:dangExpand}, with probability at least $1 - (q-1)\delta^{\kappa - \epsilon}$, we have
\[\dang(V_1,\dotsc,V_q) \geq \delta^{(q-1)} > 0\]
and hence $V_1 + \dotsb + V_q$ is a direct sum.

Let $\rho \geq \delta$ and $W \in \Gr(\R^n,r)$. By \eqref{eq:dangExpand2}, we know that if
\[\dang(V_1 + \dotsb + V_q,W) \leq \rho\]
then for some $j = 1,\dotsc,q$,
\[\dang(V_j, V_1 + \dotsb + V_{j-1} + W) \leq \rho^{\Inv{q}},\]
which happens with probability at most $\delta^{-\epsilon}\rho^\frac{\kappa}{q}$. Therefore,
\[ \Prob{\dang(V_1 + \dotsb + V_q,W) \leq \rho} \leq q\delta^{-\epsilon}\rho^\frac{\kappa}{q}.\]
Hence
\[\mu'(\Vang(W,\rho)) \leq \frac{q \delta^{-\epsilon}\rho^\frac{\kappa}{q}}{1 - (q-1)\delta^{\kappa - \epsilon}} \leq \delta^{-O(\epsilon)}\rho^\frac{\kappa}{q}.\qedhere\]
\end{proof}

\begin{proof}[Proof of Proposition~\ref{pr:nqmrcase}]
Let $n,m,q,r$ be positive integers such that $qm + r = n$. Suppose Theorem~\ref{thm:main} is true for the dimensions $n$ and $m'= qm$ but it fails for the dimensions $n$ and $m$ with parameters $0 <\alpha < n$, $\kappa > 0$ and $\epsilon > 0$. Let $A$ and $\mu$ be a counterexample, i.e. $A$ and $\mu$ satisfy \eqref{eq:sizeA}--\eqref{eq:nonconmu} but $\mu(\EC(A')) > \delta^\epsilon$ for all subsets $A' \subset A$. We will get a contradiction when $\epsilon$ is smaller than a constant depending only on $n$, $\alpha$ and $\kappa$.

Let $V_1,\dotsc V_q$ be random elements of $\Gr(\R^n,m)$ independently distributed according to $\mu$. Write $V = V_1 + \dotsb + V_q$ and let $\mu'$ be the distribution of $V$ contional to the event $\dim(V)=qm$ as in Lemma~\ref{lm:sumsOK}. It is a probability measure on $\Gr(\R^n,qm)$ satisfying a non-concentration property, according to Lemma~\ref{lm:sumsOK}. Thus, we can apply the induction hypothesis (Theorem~\ref{thm:main} combined with Proposition~\ref{pr:unionAps}) with dimensions $n$ and $m'=qm$ to the set $A$ and the measure $\mu'$. It gives $\epsilon' = \epsilon'(n,\alpha,\kappa) > 0$ such that for all $\epsilon \leq \epsilon'$, 
the probability that there exists $A' \subset A$ satisfying $\Ncov(A') \geq \delta^{\epsilon'}\Ncov(A)$ and
\[\Ncov(\pi_{V}(A')) \leq \delta^{-\frac{qm}{n}\alpha - \epsilon'}\]
is at most $\delta^{\epsilon'} + (q-1)\delta^{\kappa - \epsilon}$.

The rest of the proof consist of proving a lower bound for the same probability. First, $V_1 \in \EC(A)$ with probability at least $\delta^\epsilon$. When this happens, there is $A_1 \subset A$ with $\Ncov(A_1) \geq \delta^{\epsilon}\Ncov(A)$ and $\Ncov(\pi_{V_1}(A_1)) \leq \delta^{-\frac{m}{n}\alpha - \epsilon}$. 
Write $\epsilon_1 = \frac{3\epsilon}{\kappa}$. Then conditional to any choice of $V_1$, we have $V_2 \in \EC(A_1)\setminus \Vang(V_1,\delta^{\epsilon_1})$ with probability at least $\delta^{2\epsilon}$. When this happens, there is $A_2 \subset A_1$ with $\Ncov(A_2) \geq \delta^{\epsilon}\Ncov(A_1)$ and $\Ncov(\pi_{V_2}(A_2)) \leq \delta^{-\frac{m}{n}\alpha - \epsilon}$.
Then conditional to any choice of $V_1$ and $V_2$, the probability that $V_3 \in \EC(A_2) \setminus \Vang(V_1 + V_2, \delta^{\epsilon_1})$ is at least $\delta^{2\epsilon}$. We continue this construction until we get $A_q$.

To summarize, we have with probability at least $\delta^{(2q-1)\epsilon}$,
\[\dang(V_1,\dotsc,V_q) \geq \delta^{O(\epsilon)}\] 
and there exists a subset $A_q \subset A$ satisfying $\Ncov(A_q) \geq \delta^{q\epsilon}\Ncov(A)$ and for every $j = 1,\dotsc,q$,
\[\Ncov(\pi_{V_j}(A_q)) \leq \delta^{-\frac{m}{n}\alpha - \epsilon}\]
and hence, by Lemma~\ref{lm:tiltCart} applied to $\pi_V(A)$ in $V = \oplus_{j=1}^q V_j$, 
\[\Ncov(\pi_{V}(A_q)) \leq  \delta^{-\frac{qm}{n}\alpha - O(\epsilon)}.\]
This leads to a contradiction when $\epsilon \leq \frac{\epsilon'}{O(1)}$.
\end{proof}

\subsection{Projection to intersection of subspaces I: a discrete model}\label{ss:nqn-mrcaseI}
When the projections of a set $A$ to subspaces $V_1,\dotsc,V_q$ are all small, we would like to say that its projection to the intersection $V=V_1\cap \dotsb \cap V_q$ is small as well. This is not true. A typical example is $A = (\R e_1 \oplus \R e_2) \cup \R e_3$ where $(e_1,e_2,e_3)$ is the standard basis in $\R^3$. While its projections to $\R e_1 \oplus \R e_3$ and to $\R e_2 \oplus \R e_3$ are both small (have dimension $1$ in a $2$\dash{}dimensional space), its projection to $\R e_3$ is full dimensional. In this example, $A$ contains a large slice orthogonal to $V$. This happens to be the major obstruction. 
\begin{prop}\label{pr:SliceEndelta}
Let $n,m,q,r$ be positive integers such that $n = q(n-m) + r$. For any $0 < \alpha < n$ and $\epsilon > 0$, the following is true for sufficiently small $\delta > 0$. Let $A \subset \R^n$ and $V_1,\dotsc,V_q \in \Gr(\R^n,m)$. Write $V = V_1 \cap \dotsb \cap V_q$. Assume that
\begin{enumerate}
\item $\dang(V_1^\perp,\dotsc,V_q^\perp) \geq \delta^\epsilon$;
\item \label{it:AnotBIG} $\delta^{-\alpha + \epsilon} \leq \Ncov(A) \leq \delta^{-\alpha - \epsilon}$;
\item For every $j = 1,\dotsc,q$, $\Ncov(\pi_{V_j}(A)) \leq \delta^{-\frac{m}{n}\alpha-\epsilon}$;
\item \label{it:noBigS} For all $y \in V$, $\Ncov\bigl(A \cap \pi_V^{-1}(y^{(\delta)})\bigr) \leq \delta^{-\frac{n-r}{n}\alpha-\epsilon}$.
\end{enumerate}
Then there exists $A' \subset A$ such that $\Ncov(A') \geq \delta^{O(\epsilon)} \Ncov(A)$ and
\[\Ncov(\pi_V(A')) \leq \delta^{-\frac{r}{n}\alpha-O(\epsilon)}.\]
\end{prop}

This proposition is deduced from the following discrete analogue. Let $n,m,q,r$ be as in Proposition~\ref{pr:SliceEndelta}. For $I \subset \ensA{n}$, we write $\varpi_I \colon \Z^n \to \Z^I$ to denote the discrete projection $(z_i)_{i \in \ens{1,\dotsc,n}} \mapsto (z_i)_{i \in I}$.
Consider $I_0 = \ens{n-r +1, \dotsc,n}$ and for $j = 1,\dotsc, q$
\[I_j = \ensA{n} \setminus \ens{(j-1)(n-m) + 1, \dotsc, j(n-m)}.\]
\begin{prop}\label{pr:SliceEnergy}
We use the notations above. For any parameter $K \geq 1$ and any finite subset $Z \subset \Z^n$. One of the following statements is true.
\begin{enumerate}
\item There exists $j \in \ens{1,\dotsc,q}$ such that $\abs{\varpi_{I_j}(Z)} \geq K\abs{Z}^{\frac{m}{n}}$.
\item There exists $y \in \Z^{I_0}$ such that $\abs{Z\cap \varpi_{I_0}^{-1}(y)} \geq K\abs{Z}^{\frac{n-r}{n}}$.
\item There exists $Z' \subset Z$ such that $\abs{Z'} \geq \frac{1}{2K^{q+1}}\abs{Z}$ and $\abs{\varpi_{I_0}(Z')} \leq 2K^q\abs{Z}^{\frac{r}{n}}$.
\end{enumerate}
\end{prop}

One of the ingredients is a discrete projection inequality due to Bollob{\'a}s and Thomason~\cite{BollobasThomason} known as the uniform cover theorem. Let $\mathcal{P}(\ens{1,\dotsc, n})$ denote the set of subsets of $\ensA{n}$. Recall that a multiset of subsets of $\ens{1,\dotsc, n}$ is a collection of elements of $\mathcal{P}(\ens{1,\dotsc, n})$ which can have repeats. Giving such a multiset is equivalent to giving a map from $\mathcal{P}(\ens{1,\dotsc, n})$ to $\N$. Following Bollob\'as-Thomason, we say a multiset $\CC$ is $k$\dash{}uniform cover of $\ens{1,\dotsc,n}$ if each element $i \in \ens{1,\dotsc,n}$ belongs to exactly $k$ members of $\CC$. For exemple, with $I_j$ defined above, $(I_1\setminus I_0,\dotsc,I_q\setminus I_0)$ is a $(q-1)$\dash{}uniform cover of $\ensA{n}\setminus I_0$.

\begin{thm}[Uniform Cover theorem, Bollob\'as-Thomason~\cite{BollobasThomason}]\label{thm:UCT}
Let $Z$ be a finite subset of $\Z^n$. Let $\CC$ be an $k$\dash{}uniform cover of $\ens{1,\dotsc,n}$. Then we have
\begin{equation*}
\abs{Z}^k \leq \prod_{I \in \CC}\abs{\varpi_I(Z)}.
\end{equation*}
\end{thm}

This is a generalisation of an isoperimetric inequality due to Loomis and Whitney~\cite{LoomisWhitney}. 
For example, if we consider projections onto all canonical $m$\dash{}dimensional subspaces. There is always one which has at least the expected size: there exists $I \subset \ens{1,\dotsc,n}$ such that $\abs{I} = m$ and $\abs{\varpi_I(Z)} \geq \abs{Z}^{m/n}$.
Although the Loomis-Whitney inequality is already sufficient for the proof of Proposition~\ref{pr:SliceEnergy}, we will work at a slightly greater generality (the lemma below), since it requires no extra effort.

\begin{lemm}\label{lm:EnergyProj}
Let $I_0 \subset \ens{1,\dotsc,n}$. Let $Z$ be a finite subset of $\Z^n$ and $\mathcal{C}$ a $k$\dash{}uniform cover of $\ens{1,\dotsc,n} \setminus I_0$ with $q$ elements. Then
\begin{equation*}
\abs{Z}^{2q-k} \leq \En(\varpi_{I_0},Z)^{q-k} \prod_{I \in \CC} \abs{\varpi_{I_0 \cup I}(Z)}.
\end{equation*}
\end{lemm}
This lemma is a refinement of the uniform cover theorem. Indeed, for $I_0 = \varnothing$, we have $\En(\varpi_{I_0},Z) = \abs{Z}^2$ and we recover the uniform cover theorem.

\begin{proof}
For all $I \in \CC$, we have
\[\abs{\varpi_{I_0 \cup I}(Z)} = \sum_{y \in \varpi_{I_0}(Z)} \abs{\varpi_I(Z \cap \varpi_{I_0}^{-1}(y))}.\]
Hence, by H\"older's inequality,
\[\sum_{y \in \varpi_{I_0}(Z)} \prod_{I \in \CC} \abs{\varpi_I(Z \cap \varpi_{I_0}^{-1}(y))}^{\Inv{q}} \leq \prod_{I \in \CC} \bigl( \sum_y \abs{\varpi_I(Z \cap \varpi_{I_0}^{-1}(y))} \bigr)^{\Inv{q}} = \prod_{I \in \CC} \abs{\varpi_{I_0 \cup I}(Z)}^{\Inv{q}}.\]
For each $y \in \varpi_{I_0}(Z)$, we apply the uniform cover theorem (Theorem~\ref{thm:UCT}) to the set $Z \cap \varpi_{I_0}^{-1}(y)$ seen as a finite subset of $\Z^{\ens{1,\dotsc,n} \setminus I_0}$,
\[\abs{Z \cap \varpi_{I_0}^{-1}(y)}^{\frac{k}{q}} \leq \prod_{I \in \CC} \abs{\varpi_I(Z \cap \varpi_{I_0}^{-1}(y))}^{\Inv{q}}.\]
From the two inequalities above, we get 
\[\norm{\varpi_{I_0}{}_*\!\indic_Z}_{\frac{k}{q}}^k \leq \prod_{I \in \CC} \abs{\varpi_{I_0 \cup I}(Z)}.\]
Finally, H\"older's inequality implies
\[\abs{Z} = \norm{\varpi_{I_0}{}_*\!\indic_Z}_1 \leq \norm{\varpi_{I_0}{}_*\!\indic_Z}_{\frac{k}{q}}^{\frac{k}{2q-k}} \norm{\varpi_{I_0}{}_*\!\indic_Z}_{2}^{\frac{2q-2k}{2q-k}}.\]
We finish the proof by putting the last two inequalities together and recalling that $\En(\varpi_{I_0},Z) = \norm{\varpi_{I_0}{}_*\!\indic_Z}_{2}^2$.
\end{proof}

\begin{proof}[Proof of Proposition~\ref{pr:SliceEnergy}]
We use the notations introduced before Proposition~\ref{pr:SliceEnergy}. Applying Lemma~\ref{lm:EnergyProj} to the $(q-1)$\dash{}uniform cover $(I_1\setminus I_0,\dotsc,I_q\setminus I_0)$ of $\ensA{n}\setminus I_0$, we get
\[\abs{Z}^{q+1} \leq \En(\varpi_{I_0},Z)\prod_{j=1}^q\abs{\varpi_{I_j}(Z)}.\]
If the first statement of Proposition~\ref{pr:SliceEnergy} does not hold, we would have
\[\En(\varpi_{I_0},Z) \geq \frac{1}{K^q}\abs{Z}^{1+ \frac{n-r}{n}}.\]
If the second statement fails as well, we can apply Lemma~\ref{lm:EnergySubset} with $M = K\abs{Z}^{\frac{n-r}{n}}$ and $K' = K^{q+1}$. The third statement follows immediately.
\end{proof}

\begin{proof}[Proof of Proposition~\ref{pr:SliceEndelta}]
Let $(e_1,\dotsc,e_n)$ denote the standard basis of $\R^n$. First consider the special case where $V_j^\perp$ is exactly $\Span(e_{(j-1)(n-m) + 1},\dotsc, e_{j(n-m)})$ for each $j = 1,\dotsc,q$. Then we conclude easily from Proposition~\ref{pr:SliceEnergy} by setting $K = \delta^{-2\epsilon}$ and
\[Z = \ens{x \in \Z^n \mid A \cap \delta\cdot(x + [0,1]^n) \neq \varnothing}.\]

For the general case we consider a map $f \in \GL(\R^n)$ which sends isometrically $V$ to $\Span(e_{n-r+1},\dotsc,e_n)$ and $V_j^\perp$ to $\Span(e_{(j-1)(n-m) + 1},\dotsc, e_{j(n-m)})$ for each $j = 1,\dotsc,q$. It is easy to see that $\norm{f^{-1}}\leq n$ and
\[\abs{\det(f^{-1})} = \dang(V_1^\perp,\dotsc,V_q^\perp,V) = \dang(V_1^\perp,\dotsc,V_q^\perp).\]
Therefore $f$ is $\delta^{-O(\epsilon)}$-bi-Lipschitz. 

The conclusion for $A$ follows from the special case applied to $fA$. Indeed, by the inequality \eqref{eq:sizefA} and Lemma~\ref{lm:GLaction}, the hypotheses are satisfied for $fA$ and $f^\perp V_1,\dotsc,f^\perp V_q$ with $\epsilon$ replaced by $O(\epsilon)$. Moreover, the conclusion for $fA$ and $f^\perp V= f^\perp V_1\cap \dotsb \cap f^\perp V_q$ implies that for $A$ and $V$, again by \eqref{eq:sizefA} and Lemma~\ref{lm:GLaction}. 
\end{proof}

\subsection{Projection to intersection of subspaces II: concluding proof}\label{ss:nqn-mrcaseII}
Once we have Proposition~\ref{pr:SliceEndelta}, to prove Proposition~\ref{pr:nqn-mrcase}, we can use Proposition~\ref{pr:Nincr} and ideas in Subsection~\ref{ss:nqmcase} to rule out the case where $A$ has a very large slice and then apply the arguments in Subsection~\ref{ss:nqmrcase} to the dual.

\begin{proof}[Proof of Proposition~\ref{pr:nqn-mrcase}]
Let $n,m,q,r$ be as in Proposition~\ref{pr:nqn-mrcase}. Assume that Theorem~\ref{thm:main} is true for the dimensions $n$ and $m'=r$ and assume that $A$ and $\mu$ are counterexample to Theorem~\ref{thm:main} for the dimensions $n$ and $m$ with parameters $0< \alpha < n$, $\kappa > 0$ and $\epsilon > 0$. We begin by making two remarks. Firstly, we can assume that
\begin{equation}\label{eq:AnotBIG}
\Ncov(A) \leq \delta^{-\alpha - O(\epsilon)},
\end{equation}
for otherwise, we could conclude directly by using Proposition~\ref{pr:Nincr}.

Secondly, we can also assume that $A$ does not contain very large slice of codimension $r$. More precisely, we can assume that
\begin{equation}\label{eq:noBigS}
\forall \,W \in \Gr(\R^n,n-r),\; \forall x \in \R^n, \quad \Ncov\bigl(A \cap (x + W^{(\delta)})\bigr) \leq \delta^{-\frac{n-r}{n}\alpha - O(\epsilon)}.
\end{equation}
Indeed, if \eqref{eq:noBigS} fails, then put $B = \pi_W\bigl(A \cap (x + W^{(\delta)})\bigr)$ and we can apply Proposition~\ref{pr:Nincr} to $B \subset W$ to obtain that $\EC(B) \cap \Gr(W,m)$ does not support any measure with the corresponding non-concentration property in $\Gr(W,m)$. We can conclude as in Subsection~\ref{ss:nqmcase} by using Lemma~\ref{lm:piWVok} and Lemma~\ref{lm:piVpB}.

From now on assume \eqref{eq:AnotBIG} and \eqref{eq:noBigS}. Let $V_1,\dotsc,V_q$ be random elements of $\Gr(\R^n,m)$ independently distributed according to $\mu$. On account of \eqref{eq:perpperp}, the non-concentration property~\eqref{eq:nonconmu} implies similar property for the distribution of $V_1^\perp$, namely,
\[\forall \rho \geq \delta,\; \forall W \in \Gr(\R^n,m),\quad \Prob{\dang(V_1^\perp,W) \leq \rho} \leq \delta^{-\epsilon}\rho^\kappa.\]

From Lemma~\ref{lm:sumsOK} applied to $V_1^\perp,\dotsc, V_q^\perp$, we know that with probability at least $1 - (q-1)\delta^{\kappa - \epsilon}$, the intersection $V = V_1 \cap \dotsb \cap V_q$ has dimension $r$. Let $\mu'$ be the distribution of $V$ conditional to this event. Then by Lemma~\ref{lm:sumsOK} and \eqref{eq:perpperp}, $\mu'$ has the following non-concentration property
\[\forall \rho \geq \delta,\; \forall W \in \Gr(\R^n,n-r),\quad \Prob{\dang(V,W) \leq \rho} \leq \delta^{-O(\epsilon)}\rho^\frac{\kappa}{q}.\]

That is why we can apply the induction hypothesis (Theorem~\ref{thm:main} combined with Proposition~\ref{pr:unionAps}) to the set $A$ and the measure $\mu'$ with $n$ and $m' = r$. We obtain $\epsilon' = \epsilon'(n,\alpha,\kappa) > 0$ such that for all $\epsilon \leq \epsilon'$, the probability that there exists $A' \subset A$ satisfying
\[\Ncov(A') \geq \delta^{\epsilon'}\Ncov(A) \text{ and } \Ncov(\pi_{V}(A')) \leq \delta^{-\frac{r}{n}\alpha - \epsilon'}\]
is at most $\delta^{\epsilon'} + (q-1)\delta^{\kappa - \epsilon}$.

Now we are going to prove a lower bound for this propability, which will lead to a contradiction. As the conclusion of Theorem~\ref{thm:main} fails for $A$, we have $\mu(\EC(A'))\geq \delta^\epsilon$ for all subsets $A'\subset A$. Using a similar construction as in the proof of Proposition~\ref{pr:nqmrcase}, we prove that with probability at least $\delta^{O(\epsilon)}$, we have
\[\dang(V_1^\perp,\dotsc,V_q^\perp) \geq \delta^{O(\epsilon)}\]
and there exists $A_q \subset A$ satisfying $\Ncov(A_q) \geq \delta^{O(\epsilon)}\Ncov(A)$ and for all $j=1,\dotsc,q$,
\begin{equation*}
\pi_{V_j}(A_q) \leq \delta^{-\frac{m}{n}\alpha - \epsilon}.
\end{equation*}
Therefore, all the hypotheses of Proposition~\ref{pr:SliceEndelta} are satisfied for the set $A_q$ with $O(\epsilon)$ in the place of $\epsilon$. In particular, the assumption~\ref{it:AnotBIG} is guaranteed by \eqref{eq:sizeA} and \eqref{eq:AnotBIG} and the assumption~\ref{it:noBigS} is guaranteed by \eqref{eq:noBigS}. Hence there exists a subset $A' \subset A_q$ such that
\[\Ncov(A') \geq \delta^{O(\epsilon)}\Ncov(A) \text{ and } \Ncov(\pi_V(A')) \leq \delta^{-\frac{r}{n}\alpha - O(\epsilon)}.\]
This leads to a contradiction when $\epsilon \leq \frac{\epsilon'}{O(1)}$.
\end{proof}

%% file: hausd.tex
\section{Projection of fractal sets}
In this section we derive Theorem~\ref{cr:proj} from Theorem~\ref{thm:proj} then Corollary~\ref{cr:projFrost} and Corollary~\ref{cr:projRes} from Theorem~\ref{cr:proj}.

\subsection{Proof of Theorem~\ref{cr:proj}}
To deduce Theorem~\ref{cr:proj} from Theorem~\ref{thm:proj} we need to know how to discretize a fractal set.
The idea is the same as in the proof of \cite[Theorem 4]{Bourgain2010}. We include a detailed proof here for the sake of completeness. But before that, let us recall Frostman's lemma.

\begin{thm}[Frostman's lemma {(see \cite[Theorem 8.8]{Mattila})}]
\label{thm:Frostman}
Let $A$ be a Borel set of $\R^n$. If $\dimH(A) > \alpha$ then there exists a finite nonzero compactly supported Borel measure $\nu$ with $\Supp(\nu) \subset A$ such that
\[\forall \rho > 0,\; \forall x \in \R^n,\quad \nu(\Ball(x,\rho)) \leq \rho^\alpha.\] 
\end{thm}

\begin{proof}[Proof of Theorem~\ref{cr:proj}]
Let $0 < m < n$, $0 < \alpha < n$, $\kappa > 0$ be parameters. Let $\epsilon> 0$ be $\frac{1}{4}$ times the constant given by Theorem~\ref{thm:proj} applied to these parameters. Let $A$ and $\mu$ be a counterexample for Theorem~\ref{cr:proj} with these parameters. Without loss of generality we can assume $A \subset \Ball(0,1)$.

After normalizing $\mu$, we can suppose that it is a probability measure such that
\[ \forall \rho > 0,\; \forall W \in \Gr(\R^n,n-m),\quad \mu(\Vang(W,\rho)) \ll_\mu \rho^\kappa.\]
Thus, the non-concentration condition~\eqref{eq:nonconmu} of Theorem~\ref{thm:proj} is satisfied for sufficiently small $\delta$.

By Frostman's lemma, there is a nonzero Radon measure $\nu$ compactly supported on $A$ such that
\begin{equation}\label{eq:nuFrost}
\forall \rho > 0,\; \forall x \in \R^n, \quad \nu(\Ball(x,\rho)) \leq \rho^{\alpha - \epsilon}.
\end{equation}
For any $V \in \Supp(\mu)$ we have $\dimH(\pi_V(A)) < \eta$ where $\eta = \frac{m}{n} \alpha + 2\epsilon$. By the definition of Hausdorff dimension, for any $k_0 \geq 1$, there is a cover 
\[\pi_V(A) \subset \bigcup_{k \geq k_0} B_{V,k}\]
of $\pi_V(A)$ such that each $B_{V,k}$ is a union of at most $2^{k\eta}$ balls of radius $2^{-k}$ in $V$.

Set $A_{V,k} = \pi_V^{-1}(B_{V,k})$ for $V \in \Supp(\mu)$ and $k \geq k_0$. Since the sets $A_{V,k}$, $k \geq k_0$,  cover $A$, we have
\[\sum_{k \geq k_0} \nu(A_{V,k}) \gg_\nu 1 .\]
Integrating with respect to $\dd \mu (V)$ and using Fubini's theorem, we obtain
\[\sum_{k \geq k_0} \int \nu(A_{V,k}) \dd\mu(V) \gg_{\nu} 1.\]
This in turn implies that there exists $k \geq k_0$ such that $\mu(\EC) \gg_\nu k^{-2}$, where 
\[\EC = \{ V \in \Gr(\R^n,m) \mid \nu(A_{V,k}) \gg_\nu k^{-2}\}.\]
Now fix this $k$ and set $\delta = 2^{-k}$ so that $\Ncov(\pi_V(A_{V,k})) \leq \delta^{-\eta}$. Note that as we can choose $k_0$ arbitrarily large, we can make $\delta$ arbitrarily small.

Here we cannot apply Theorem~\ref{thm:proj} directly to the set $A$ because it might not be regular enough. The idea is to partition $A$ into regular parts. Let $\QC$ denotes the set of dyadic cubes in $\R^n$ of side length $\delta$:
\[ \QC = \bigl\{x + {[0,\delta[}^n \;\mid x \in \delta \cdot \Z^n\bigr\}.\]
Put $L = \ceil{\frac{n}{\epsilon}} + 1$. For $l = 0, \dotsc , L$, let $A_l$ be the union of all cubes $Q \in \QC$ such that
\[\delta^{(l+1)\epsilon}\nu(A) < \nu(Q) \leq \delta^{l\epsilon}\nu(A).\]

It is easy to see that $A_l$ are disjoint and $\sum_{l=0}^L \nu(A_l) \geq (1 - \delta^\epsilon) \nu(A)$. Moreover for any $l = 0,\dotsc,L$ and any $A'\subset A_l$ which is also a union of cubes in $\QC$, we have
\[\delta^{(l+1)\epsilon}\Ncov(A')\nu(A) \ll_n \nu(A') \ll_n \delta^{l\epsilon}\Ncov(A')\nu(A)\]
Hence, if $\nu(A_l) > 0$, then for such $A'$,
\begin{equation}\label{eq:nuNcov}
 \delta^{\epsilon}\frac{\nu(A')}{\nu(A_l)} \ll_n \frac{\Ncov(A')}{\Ncov(A_l)} \ll_n \delta^{-\epsilon}\frac{\nu(A')}{\nu(A_l)}
\end{equation}

Consider $\LC = \ens{0 \leq l \leq L \mid \nu(A_l) \geq \delta^\epsilon}$, the set of levels with sufficient mass. For any $l \in \LC$, by \eqref{eq:nuFrost},
\[ \Ncov(A_l) \gg \delta^{-\alpha+\epsilon} \nu(A_l) \geq \delta^{-\alpha+2\epsilon}\]
and from \eqref{eq:nuNcov} and \eqref{eq:nuFrost}, for any $\rho \geq \delta$ and any $x \in \R^n$,
\[\frac{\Ncov(A_l \cap \Ball(x,\rho))}{\Ncov(A_l)} \ll_n \delta^{-\epsilon} \frac{\nu(\Ball(x,\rho + n \delta))}{\nu(A_l)} \ll_n \delta^{-3\epsilon} \rho^\alpha.\]
In other words, the assumptions of Theorem~\ref{thm:proj} are satisfied for $A_l$.

Now for $l \in \LC$ and $V \in \EC$, let $A_{V,k,l}$ be the union of $Q \in \QC$ such that $Q \subset A_l$ and $Q \cap A_{V,k} \neq \varnothing$. From the definition of $\LC$ and $\EC$, we know that for any $V \in \EC$
\[\sum_{l\in \LC} \nu(A_{V,k,l}) \geq \nu(A_{V,k}) - (L+1)\delta^\epsilon \gg_\nu k^{-2}.\]
Hence there exists $l \in \LC$ such that $\frac{\nu(A_{V,k,l})}{\nu(A_l)} \gg_\nu k^{-2}$. Therefore by setting 
\[\EC_l = \Bigl\{ V \in \Gr(\R^n,m) \mid \frac{\nu(A_{V,k,l})}{\nu(A_l)} \gg_{\nu} k^{-2}\Bigr\},\]
we have $\EC = \cup_{l\in \LC} \EC_l$.

From the lower bound $\mu(\EC) \gg_\nu k^{-2}$, we find a certain $l \in \LC$ such that $\mu(\EC_l) \gg_{\nu,L} k^{-2}$. This contradicts Theorem~\ref{thm:proj} applied to the set $A_l$ and the measure $\mu$. Indeed, recalling the notation \eqref{eq:ECAeps}, Theorem~\ref{thm:proj} says $\mu(\EC(A_l,4\epsilon)) \leq \delta^{4\epsilon}$. But we have $\EC_l \subset \EC(A_l,4\epsilon)$. Because for all $V \in \EC_l$, we have $\Ncov(A_{V,k,l}) \geq \delta^{2\epsilon} \Ncov(A_l)$ by \eqref{eq:nuNcov} and also 
\[\Ncov(\pi_V(A_{V,k,l})) \ll \Ncov(\pi_V(A_{V,k})) \leq \delta^{-\eta}. \qedhere\]
\end{proof}

\subsection{Hausdorff dimension of exceptional set}
In this subsection we deduce Corollary~\ref{cr:projFrost} from Theorem~\ref{cr:proj}. First recall the \L{}ojasiewicz inequality which we will need.  
\begin{thm}[\L{}ojasiewicz inequality {\cite[Théorème 2, page 62]{Lojasiewicz}}]
Let $(M,d)$ be a real analytic manifold endowed with a Riemannian distance $d$ and let $f \colon M \to \R$ be a real analytic map. If $K$ is a compact subset of $M$, then there is $C > 0$ depending on $K$ and $f$ such that for all $x \in K$,
\[\abs{f(x)} \geq \Inv{C}\min(1,d(x,Z))^C\]
where $Z = \ens{x \in M \mid f(x) = 0}$.
\end{thm}

\begin{proof}[Proof of Corollary~\ref{cr:projFrost}]
Recall that we work with a Riemannian metric on the Grassmannian $\Gr(\R^n,m)$ which is invariant under the action of the group $\grO(n)$.
Observe that the exceptional set of directions
\[\{ V\in \Gr(\R^n,m) \mid \dimH(\pi_V(A)) \leq \frac{m}{n}\alpha + \epsilon\}\]
is measurable for the Borel $\sigma$-algebra on $\Gr(\R^n,m)$. Suppose that the Hausdorff dimension of the exceptional set is larger than  $m(n-m) - 1 + \kappa$ for some $\kappa$. Frostman's lemma is valid for general compact metric spaces (see \cite[Theorem 8.17]{Mattila})\footnote{For our situation, we can simply use a local chart and Frostman's lemma in $\R^n$ since a local chart of a Riemannian manifold is necessarily bi-Lipschitz to its image (endowed with the induced Euclidean distance).}. Thus there exists a nonzero Radon measure $\mu$ supported on this exceptional set such that for all $\rho > 0$ and all $V \in \Gr(\R^n,m)$, $\mu(\Ball(V,\rho)) \leq \rho^{m(n-m) - 1 + \kappa}$. We are going to prove that $\mu$ satisfies the non-concentration property forbidden by Theorem~\ref{cr:proj}.

We fix $W \in \Gr(\R^n,m)$ and apply the \L{}ojasiewicz inequality to the real analytic function $\dang(\bullet,W)^2 \colon \Gr(\R^n,m) \to \R$. We conclude that there is a constant $C > 0$ such that for any $0 < \rho \leq 1$, $\Vang(W,\rho)$ is contained in the $\rho'$\dash{}neighborhood of the Schubert cycle $\Vang(W,0)$ with $\rho' = (C\rho)^{\Inv{C}}$. By the $\grO(n)$\dash{}invariance, the constant $C$ is in fact uniform for all $W \in \Gr(\R^n,n-m)$. Since the Schubert cycle $\Vang(W,0)$ is a smooth submanifold, we have $\Ncov[\rho'](\Vang(W,0)^{(\rho')}) \ll_n \rho'^{-m(n-m) + 1}$. Here again, the estimate is uniform in $W$ thanks to the $\grO(n)$\dash{}invariance. Therefore,
\[\mu(\Vang(W,\rho)) \leq \Ncov[\rho'](\Vang(W,0)^{(\rho')}) \sup_{V \in \Gr(\R^n,m)} \mu(\Ball(V,\rho')) \ll_n \rho^{\frac{\kappa}{C}}.\]
This contradicts Theorem~\ref{cr:proj} if $\epsilon$ is sufficiently small and finishes the proof of Corollary~\ref{cr:projFrost}.
\end{proof}

\subsection{Restricted family of projections}
Finally, we deduce Corollary~\ref{cr:projRes} from Theorem~\ref{cr:proj}. We will use the following Remez-type inequality due to A. Brudnyi~\cite{Brudnyi2010}.

\begin{thm}[Brudnyi {\cite[Theorem 1.2, case (b)]{Brudnyi2010}}]\label{thm:Brudnyi}
Let $\Omega_1 \subset \R^q$ and $\Omega_2 \subset \R^p$ be connected open sets in Euclidean spaces. Let $f \colon \Omega_1 \times \Omega_2 \to \R$ be a real analytic map, considered as a family of real analytic functions on $\Omega_2$ depending analytically on a parameter varying in $\Omega_1$. Let $K_1 \subset \Omega_1$ and $K_2 \subset \Omega_2$ be compact subsets. For $x \in \Omega_1$, write
\[M(x) = \max_{y \in K_2} \abs{f(x,y)}.\]
There exists $d \geq 1$ and $C > 0$ such that for any $0 < \rho < 1$, 
\begin{equation}\label{eq:Remez}
 \max_{x \in K_1}\,\Ncov[\rho](\{y \in K_2 \mid \abs{f(x,y)} < M(x)\rho^d \}) \leq C \rho^{1-p}.
\end{equation}
\end{thm}

The constant $d$ in \eqref{eq:Remez} depends not only on $f$ but also on $K_1$ and $K_2$. It can be estimated if more is known about the function $f$. For example, if for any $x \in \Omega_1$, $f(x,\bullet) \colon \Omega_2 \to \R$ is polynomial of degree less than $k \in \N$, then $d$ can be taken to be $k$. We refer the reader to the introduction in \cite{Brudnyi2010} for more details.

The function $M$ is lower semicontinuous. Hence by compactness $M$ has a minimum on $K_1$. As a consequence, if we assume that for any $x \in K_1$, we have $M(x) > 0$, i.e. $f(x,\bullet)$ is not identically zero, then the conclusion~\eqref{eq:Remez} can be reformulated as (the constant $C$ becomes larger in this formulation)
\begin{equation*}
\max_{x \in K_1}\,\Ncov[\rho](\{y \in K_2 \mid \abs{f(x,y)} \leq \rho^d \}) \leq C \rho^{1-p}.
\end{equation*}
Moreover, by covering with charts, it is easy to see that the same holds if $\Omega_1$ and $\Omega_2$ are connected real analytic manifolds.

\begin{proof}[Proof of Corollary~\ref{cr:projRes}]
Let $\epsilon > 0$ as be given by Theorem~\ref{cr:proj}. Consider the real analytic map $f \colon \Gr(n,n-m) \times \Omega \to \R$ defined by
\[\forall (W,t) \in \Gr(n,n-m) \times \Omega, \quad f(W,t) = \dang(V(t),W)^2.\]
By the transversality assumption, for any $W$, the partial function $f(W,\bullet)$ is not identically zero. Hence by Brudnyi's theorem, there exists $d$ and $C > 0$ such that for all $ 0 < \rho < 1$,
\begin{equation}\label{eq:byRemez}
\max_{W \in \Gr(n,n-m)} \Ncov[\rho](\{t \in \Omega' \mid \abs{f(W,t)} \leq \rho^d \}) \leq C \rho^{1-p}.
\end{equation}
Now assume for a contradiction that the set of exceptional parameters has Hausdorff dimension larger than $p - 1 + d\kappa$. Then by Frostman's lemma, there exists a nonzero Borel measure $\mu$ supported on this exceptional set satisfying
\[\sup_{t \in \R^p} \mu(\Ball(t,\rho)) \leq \rho^{p - 1 + d\kappa}.\]
Then by \eqref{eq:byRemez}, for any $W \in \Gr(\R^n,n-m)$,
\[ \mu(\{t \in \Omega' \mid \dang(V(t),W) \leq \rho^d \}) \leq \rho^{p - 1 + d\kappa} C \rho^{1-p} \leq C \rho^{d\kappa}\]
In other words the image measure $\frac{1}{C}V_*\mu$ has the non-concentration property forbidden by Theorem~\ref{cr:proj}. This concludes the proof of \eqref{eq:projRes}.

The moreover part follows from the fact we know the exact value of $d$ in Theorem~\ref{thm:Brudnyi} when the map $f$ is polynomial.
\end{proof}